\documentclass[article,final]{siamart0516}
\usepackage{amsfonts}
\usepackage{pgfplots, pgfplotstable, booktabs, colortbl, array}
\usepackage{upquote}
\usepackage{graphics}

\DeclareMathOperator*{\argmin}{arg\,min}

\begin{document}

\title{Computing the Fr\'{e}chet Derivative of the Polar Decomposition\thanks{Submitted to the editors August 15, 2016.
\funding{EG has been supported in part by NSF under grants DMS-1411792, DMS-1345013. ML has been supported in part by NSF under grants DMS-1010687, CMMI-1029445, DMS-1065972, CMMI-1334759, DMS-1411792, DMS-1345013.}}}

\author{
  Evan S. Gawlik\thanks{Department of Mathematics, University of California, San Diego
    (\email{egawlik@ucsd.edu}, \email{mleok@math.ucsd.edu}).}
  \and
  Melvin Leok\footnotemark[2]
}

\headers{Computing the Fr\'{e}chet Derivative of the Polar Decomposition}{E. S. Gawlik and M. Leok}

\maketitle

\begin{abstract}
We derive iterative methods for computing the Fr\'{e}chet derivative of the map which sends a full-rank matrix $A$ to the factor $U$ in its polar decomposition $A=UH$, where $U$ has orthonormal columns and $H$ is Hermitian positive definite.  The methods apply to square matrices as well as rectangular matrices having more rows than columns.  Our derivation relies on a novel identity that relates the Fr\'{e}chet derivative of the polar decomposition to the matrix sign function $\mathrm{sign}(X) = X (X^2)^{-1/2}$ applied to a certain block matrix $X$.  
\end{abstract}



\section{Introduction}

The polar decomposition theorem asserts that every matrix $A \in \mathbb{C}^{m \times n}$ ($m \ge n$) can be written as the product $A=UH$ of a matrix $U \in \mathbb{C}^{m \times n}$ having orthonormal columns times a Hermitian positive definite matrix $H \in \mathbb{C}^{n \times n}$~\cite[Theorem 8.1]{higham2008functions}.  If $A$ is full-rank, this decomposition is unique, allowing one to define a map $\mathcal{P}$ which sends a full-rank matrix $A \in \mathbb{C}^{m \times n}$ to the factor $\mathcal{P}(A) = U  \in \mathbb{C}^{m \times n}$ in its polar decomposition $A=UH$.  We refer to $U$ as the unitary factor in the polar decomposition of $A$, bearing in mind that this is a slight abuse of terminology when $A$ (and hence $U$) is rectangular.   The aim of this paper is to derive iterative algorithms for computing the Fr\'{e}chet derivative of $\mathcal{P}$.

Our interest in differentiating the polar decomposition stems from several sources.  First, differentiating the polar decomposition gives precise information about the sensitivity of the polar decomposition to perturbations.  This is a topic of longstanding interest in numerical analysis~\cite{li1997relative,li2003new,bhatia1994matrix,kenney1991polar,mathias1993perturbation}, where much of the literature has focused on bounding the deviations in the perturbed factors in the polar decomposition of $A$ after a small-normed perturbation of $A$.  
These analyses often rely on a formula for the Fr\'{e}chet derivative of $\mathcal{P}$ that involves the singular value decomposition of $A$~\cite[Equation 2.18]{kenney1991polar}.
While theoretically useful, such a formula loses some of its appeal in the numerical setting, where computing the singular value decomposition tends to be costly.
As a second source of motivation, differentiating the polar decomposition has proven necessary in the design of
certain schemes for interpolating functions which take values in the special orthogonal group~\cite{gawlik2016embedding}, the group of real square matrices with orthonormal columns and positive determinant.  These interpolation schemes have applications in computer animation, mechanics, and other areas in which continuously varying rotation matrices play a role.

A number of authors have addressed the computation of the Fr\'{e}chet derivatives of other functions of matrices, such as the matrix exponential~\cite{al2009computing,mathias1992evaluating,najfeld1995derivatives}, the matrix logarithm~\cite{al2013computing,kenney1998schur}, the matrix square root~\cite[Section 2]{al2009computing}, the matrix $p^{th}$ root~\cite{higham2013improved,cardoso2012computation,cardoso2011evaluating}, and the matrix sign function $\mathrm{sign}(X) = X (X^2)^{-1/2}$~\cite{kenney1991polar}.  The aforementioned functions, unlike the map $\mathcal{P}$, are examples of \emph{primary matrix functions}.  Roughly speaking, a primary matrix function is a scalar function that has been extended to square matrices in a canonical way; for a precise definition, see~\cite[Section 1.2]{higham2008functions} and~\cite{horn2012matrix}.  The polar decomposition is not a primary matrix function, which is perhaps the main reason that the computation of its Fr\'{e}chet derivative has largely evaded scrutiny until now.

Formally, iterative schemes for computing the Fr\'{e}chet derivatives of matrix functions (be they primary or nonprimary) can be derived as follows.  Let $f : \mathbb{C}^{m \times n} \rightarrow \mathbb{C}^{m \times n}$ be a function with Fr\'{e}chet derivative $L_f$.  That is, given $X \in \mathbb{C}^{m \times n}$, the map $L_f(X,\cdot) : \mathbb{C}^{m \times n} \rightarrow \mathbb{C}^{m \times n}$ is a linear map satisfying
\begin{equation} \label{Frechet}
f(X+E) - f(X) - L_f(X,E) = o(\|E\|)
\end{equation}
for every $E \in \mathbb{C}^{m \times n}$, where $\|\cdot\|$ denotes any matrix norm.
Let $A \in \mathbb{C}^{m \times n}$, and suppose that
\begin{equation} \label{gscheme}
X_{k+1} = g(X_k), \quad X_0 = A
\end{equation}
is an iterative scheme for computing $f(A)$; that is, $X_k \rightarrow f(A)$ as $k \rightarrow \infty$.  Differentiation of~(\ref{gscheme}) with respect to $A$ in the direction $E \in \mathbb{C}^{m \times n}$ yields the coupled iteration
\begin{alignat}{3}
X_{k+1} &= g(X_k), &\quad& X_0 = A, \label{Xupdate0} \\
E_{k+1} &= L_g(X_k,E_k), &\quad& E_0 = E, \label{Eupdate0}
\end{alignat}
for computing $f(A)$ and $L_f(A,E)$.  The validity of this formal derivation, of course, depends on the commutativity of $\lim_{k \rightarrow \infty}$ with differentiation, which is generally nontrivial to establish.

For a primary matrix function $f$, proving the validity of this formal derivation is greatly simplified by the following identity.  For any primary matrix function $f$ and any square matrices $A$ and $E$,
\begin{equation} \label{mathias}
f \begin{pmatrix} A & E \\ 0 & A \end{pmatrix} = \begin{pmatrix} f(A) & L_f(A,E) \\ 0 & f(A) \end{pmatrix},
\end{equation}
provided that $f$ is $2p-1$ times continuously differentiable on an open subset of $\mathbb{C}$ containing the spectrum of $A$, where $p$ is the size of the largest Jordan block of $A$~\cite{mathias1996chain}.
From this it follows that if~(\ref{gscheme}) is an iterative scheme for computing $f(A)$, and if $g$ maps block upper triangular matrices to block upper triangular matrices, then
\begin{equation} \label{XEupdate}
\begin{pmatrix} X_{k+1} & E_{k+1} \\ 0 & X_{k+1} \end{pmatrix} = g\begin{pmatrix} X_k & E_k \\ 0 & X_k \end{pmatrix}, \quad \begin{pmatrix} X_0 & E_0 \\ 0 & X_0 \end{pmatrix} = \begin{pmatrix} A & E \\ 0 & A \end{pmatrix}
\end{equation}
defines an iterative scheme for computing $\begin{pmatrix} f(A) & L_f(A,E) \\ 0 & f(A) \end{pmatrix}$,
provided that it converges and provided that $f$ has the requisite regularity to apply~(\ref{mathias}).  Using~(\ref{mathias}) again to isolate each block of the iteration~(\ref{XEupdate}), one obtains the coupled iteration~(\ref{Xupdate0}-\ref{Eupdate0}).  Details behind this argument, as well as an example of its application, can be found in~\cite[Section 2]{al2009computing}.

Our main result in this paper, Theorem~\ref{thm:dpolariter}, establishes the validity of schemes like~(\ref{Xupdate0}-\ref{Eupdate0}) when the function $f$ under consideration is the function $\mathcal{P}$ which sends $A$ to the unitary factor $U$ in its polar decomposition, \emph{even though $\mathcal{P}$ is not a primary matrix function}.  In particular,  
\[
\mathcal{P} \begin{pmatrix} A & E \\ 0 & A \end{pmatrix} \neq \begin{pmatrix} \mathcal{P}(A) & L_{\mathcal{P}}(A,E) \\ 0 & \mathcal{P}(A) \end{pmatrix},
\]
so the argument in the preceding paragraph does not apply.
Instead, our derivation relies on a novel identity that relates the Fr\'{e}chet derivative of $\mathcal{P}$ to the matrix sign function $\mathrm{sign}(X) = X (X^2)^{-1/2}$ applied to a certain block matrix $X$; see Theorem~\ref{thm:signblockmat}.

One notable corollary of Theorem~\ref{thm:dpolariter} is that the popular Newton iteration~\cite{higham1986computing}
\[
X_{k+1} = \frac{1}{2}(X_k + X_k^{-*}), \quad X_0 = A
\]
for computing the unitary factor $\mathcal{P}(A)=U$ in the polar decomposition $A=UH$ of a square matrix $A$ extends to a coupled iteration for computing $\mathcal{P}(A)$ and its Fr\'{e}chet derivative.
In particular, Corollary~\ref{cor:Newton} shows that for any nonsingular $A \in \mathbb{C}^{n \times n}$ and any $E \in \mathbb{C}^{n \times n}$, the scheme
\begin{alignat}{3}
X_{k+1} &= \frac{1}{2}(X_k + X_k^{-*}), &\quad& X_0 = A, \label{XupdateNewton0} \\
E_{k+1} &= \frac{1}{2}(E_k - X_k^{-*} E_k^* X_k^{-*}), &\quad& E_0 = E, \label{EupdateNewton0}
\end{alignat}
produces iterates $X_k$ and $E_k$ that converge to $\mathcal{P}(A)=U$ and $L_{\mathcal{P}}(A,E)$, respectively, as $k \rightarrow \infty$.

The fact that the matrix sign function will play a role in our study of Fr\'{e}chet derivatives of the polar decomposition should come as no surprise, given the sign function's intimate connection with the polar decomposition.  The sign function and polar decomposition are linked via the identity
\begin{equation} \label{signpolar}
\mathrm{sign}\begin{pmatrix} 0 & A \\ A^* & 0 \end{pmatrix} = \begin{pmatrix} 0 & \mathcal{P}(A) \\ \mathcal{P}(A)^* & 0 \end{pmatrix},
\end{equation}
which holds for any square nonsingular matrix $A$~\cite{higham2008functions}.  This identity has been used, among other things, to derive iterative schemes for computing the polar decomposition.  The essence of this approach is to write down an iterative scheme for computing $\mathrm{sign}\begin{pmatrix} 0 & A \\ A^* & 0 \end{pmatrix}$, check that its iterates retain the relevant block structure, and read off the $(1,2)$-block of the resulting algorithm.  In principle, one can adopt a similar strategy to derive iterative schemes for computing the Fr\'{e}chet derivatives of the polar decomposition.  Indeed, any iterative scheme that computes 
\[
\mathrm{sign}
\begin{pmatrix} 
0 & A & 0 & E \\ 
A^* & 0 & E^* & 0 \\ 
0 & 0 & 0 & A \\
0 & 0 & A^* & 0 
\end{pmatrix}
\]
while retaining its block structure will suffice, owing to the following observation.  By appealing to the definition~(\ref{Frechet}) of the Fr\'{e}chet derivative, the identity~(\ref{signpolar}) can be used to verify that
\begin{equation} \label{Lsign}
 L_\mathrm{sign} \left( \begin{pmatrix} 0 & A \\ A^* & 0 \end{pmatrix}, \begin{pmatrix} 0 & E \\ E^* & 0 \end{pmatrix} \right) = \begin{pmatrix} 0 & L_{\mathcal{P}}(A,E) \\ L_{\mathcal{P}}(A,E)^* & 0 \end{pmatrix}.
\end{equation}
Now since the sign function is a primary matrix function,~(\ref{mathias}),~(\ref{signpolar}), and~(\ref{Lsign}) imply that
\begin{align*}
\mathrm{sign}
\begin{pmatrix} 
0 & A & 0 & E \\ 
A^* & 0 & E^* & 0 \\ 
0 & 0 & 0 & A \\
0 & 0 & A^* & 0 
\end{pmatrix}
&= \begin{pmatrix} 
\mathrm{sign}\begin{pmatrix} 0 & A \\ A^* & 0 \end{pmatrix} & L_\mathrm{sign} \left( \begin{pmatrix} 0 & A \\ A^* & 0 \end{pmatrix}, \begin{pmatrix} 0 & E \\ E^* & 0 \end{pmatrix} \right) \\
0 & \mathrm{sign}\begin{pmatrix} 0 & A \\ A^* & 0 \end{pmatrix}
\end{pmatrix} \\
&= \begin{pmatrix} 
0 & \mathcal{P}(A) & 0 & L_{\mathcal{P}}(A,E) \\ 
\mathcal{P}(A)^* & 0 & L_{\mathcal{P}}(A,E)^* & 0 \\ 
0 & 0 & 0 & \mathcal{P}(A) \\
0 & 0 & \mathcal{P}(A)^* & 0 
\end{pmatrix}.
\end{align*}
A drawback of this approach is that it is valid only for square matrices $A$.  The strategy we adopt in the present paper will be quite different, and will be valid not just for square matrices $A$ but also for rectangular matrices $A$ having more rows than columns.

\paragraph{Organization} This paper is organized as follows.  We begin in Section~\ref{sec:results} by giving statements of our main results, deferring their proof to Section~\ref{sec:proofs}.  In Section~\ref{sec:practical}, we discuss several practical aspects of the iterative schemes, including stability, scaling, and termination criteria.  We compare the iterative schemes to other methods for computing the Fr\'{e}chet derivative of the polar decomposition in Section~\ref{sec:comparison}.  We finish with some numerical experiments in Section~\ref{sec:numerical}.

\section{Statement of Results} \label{sec:results}

In this section, we give a presentation of this paper's main result, which is a theorem that details a class of iterative schemes for computing the Fr\'{e}chet derivative $L_{\mathcal{P}}$ of the map $\mathcal{P}$ which sends a matrix $A$ to the unitary factor $U$ in its polar decomposition $A=UH$.  A proof of the theorem is given in Section~\ref{sec:proofs}.

The class of iterative schemes to be considered comprises schemes of the form~(\ref{Xupdate0}-\ref{Eupdate0}), with a mild constraint on the form of the function $g$.  To understand this constraint, it is helpful to develop some intuition concerning iterative schemes for computing the polar decomposition and their relationship to iterative schemes for computing the matrix sign function.  Fundamental to that intuition are the identities
\begin{equation} \label{signpolardef}
\mathrm{sign}(A) = A (A^2)^{-1/2}, \quad \mathcal{P}(A) = A (A^* A)^{-1/2},
\end{equation}
and the integral representation formulas~\cite[Equations 6.2 and 6.3]{higham1994matrix}
\[
\mathrm{sign}(A) = \frac{2}{\pi} A \int_0^\infty (t^2 I + A^2)^{-1} \, dt, \quad \mathcal{P}(A) = \frac{2}{\pi} A \int_0^\infty (t^2 I + A^* A)^{-1} \, dt,
\]
which hint at two rules of thumb.  First, iterative schemes for computing the matrix sign function tend to have the form $X_{k+1} = X_k h(X_k^2)$, where $h$ is a primary matrix function.  Second, to each iterative scheme $X_{k+1} = X_k h(X_k^2)$ for computing the matrix sign function, there corresponds an iterative scheme $X_{k+1} = X_k h(X_k^* X_k)$ for computing the polar decomposition.  The first of these rules of thumb appears to hold empirically to our knowledge.  The second is made precise in~\cite[Theorem 8.13]{higham2008functions}.  The theorem below extends~\cite[Theorem 8.13]{higham2008functions} by showing, in essence, that to each iterative scheme $X_{k+1} = X_k h(X_k^2)$ for computing the matrix sign function, there corresponds an iterative scheme for computing the polar decomposition \emph{and its Fr\'{e}chet derivative}.  This iterative scheme is given by~(\ref{Xupdate0}-\ref{Eupdate0}) with $g(X) = X h(X^* X)$.

In what follows, we denote by $\mathrm{skew}(B) = \frac{1}{2}(B-B^*)$ and $\mathrm{sym}(B) = \frac{1}{2}(B+B^*)$ the skew-Hermitian and Hermitian parts, respectively, of a square matrix $B$.  We denote the spectrum of $B$ by $\Lambda(B)$.

\begin{theorem} \label{thm:dpolariter}
Let $A \in \mathbb{C}^{m \times n}$ ($m \ge n$) be a full-rank matrix having polar decomposition $A=UH$, where $U \in \mathbb{C}^{m \times n}$ has orthonormal columns and $H \in \mathbb{C}^{n \times n}$ is Hermitian positive definite.  Let $E \in \mathbb{C}^{m \times n}$, and define $\Omega = \mathrm{skew}(U^* E)$ and $S = \mathrm{sym}(U^* E)$.  Let $h$ be a primary matrix function satisfying $h(Z^*)=h(Z)^*$ for every $Z$, and suppose that the iteration $Z_{k+1} = Z_k h(Z_k^2)$ produces iterates $Z_k$ that converge to $\mathrm{sign}(Z_0)$ as $k \rightarrow \infty$ when the initial condition is
\begin{equation} \label{Z0}
Z_0 = \begin{pmatrix} H & \Omega \\ 0 & -H \end{pmatrix},
\end{equation}
as well as when the initial condition is
\begin{equation} \label{Z0b}
Z_0 = \begin{pmatrix} H & S \\ 0 & H \end{pmatrix}.
\end{equation}
Assume that in both cases, $h$ is smooth on an open subset of $\mathbb{C}$ containing $\cup_{k=0}^\infty \Lambda(Z_k)$.
Let $g(X) = X h(X^* X)$.  Then the iteration
\begin{alignat}{3}
X_{k+1} &= g(X_k), &\quad& X_0 = A, \label{Xupdate2} \\
E_{k+1} &= L_g(X_k,E_k), &\quad& E_0 = E, \label{Eupdate2}
\end{alignat}
produces iterates $X_k$ and $E_k$ that converge to $\mathcal{P}(A)=U$ and $L_{\mathcal{P}}(A,E)$, respectively, as $k \rightarrow \infty$.
\end{theorem}

\paragraph{Remark} 
Taking $E=0$ in the preceding theorem, one recovers~\cite[Theorem 8.13]{higham2008functions}, up to the following modification: Instead of requesting that $h$ is a primary matrix function satisfying $h(Z^*)=h(Z)^*$,~\cite[Theorem 8.13]{higham2008functions} makes the weaker assumption that the function $\widetilde{g}(Z) = Z h(Z^2)$ satisfies $\widetilde{g}(Z^*) = \widetilde{g}(Z)^*$ for every $Z$.  It is easily checked using elementary properties of primary matrix functions~\cite[Theorem 1.13]{higham2008functions} that the latter is implied by the former.

Note that it is sometimes the case that the convergence of the matrix sign function iteration $Z_{k+1} = Z_k h(Z_k^2)$ referenced in Theorem~\ref{thm:dpolariter} is dictated by the spectrum of $Z_0$.  If this is the case, then the hypothesis that the iteration converges when $Z_0$ is given by~(\ref{Z0}) or~(\ref{Z0b}) is equivalent to the simpler hypothesis that the iteration converges when $Z_0=H$.  This follows from the fact that the eigenvalues of~(\ref{Z0}) or~(\ref{Z0b}) coincide with those of $H$.

Central to the proof of Theorem~\ref{thm:dpolariter} is an identity that relates the Fr\'{e}chet derivative of the polar decomposition to the sign of the block matrix $Z_0$ appearing in~(\ref{Z0}). 
We state the identity below to emphasize its importance.  A proof is given in Section~\ref{sec:identities}.
\begin{theorem} \label{thm:signblockmat}
Let $A$, $U$, $H$, $E$, and $\Omega$ be as in Theorem~\ref{thm:dpolariter}.  Then
\begin{align}
\mathrm{sign}\begin{pmatrix} H & \Omega \\ 0 & -H \end{pmatrix} &= \begin{pmatrix} I & U^* L_{\mathcal{P}}(A,E) \\ 0 & -I \end{pmatrix}. \label{signHW0} 
\end{align}
In particular, if $U^*E$ is skew-Hermitian, then
\begin{equation} \label{signblockmat2}
\mathrm{sign}\left( \begin{pmatrix} U^* & 0 \\ 0 & -U^* \end{pmatrix} \begin{pmatrix} A & E \\ 0 & A \end{pmatrix} \right) = \begin{pmatrix} U^* & 0 \\ 0 & -U^* \end{pmatrix}  \begin{pmatrix} \mathcal{P}(A) & L_{\mathcal{P}}(A,E) \\ 0 & \mathcal{P}(A) \end{pmatrix}.
\end{equation}
\end{theorem}

In addition to being useful in the proof of Theorem~\ref{thm:dpolariter}, the identity~(\ref{signblockmat2}) bears an interesting resemblance to~(\ref{mathias}).

Theorem~\ref{thm:dpolariter} has several corollaries, each corresponding to a different choice of iterative scheme $Z_{k+1} = Z_k h(Z_k^2)$ for computing the matrix sign function.  The simplest is the well-known Newton iteration 
\begin{equation} \label{signNewton}
Z_{k+1} = \frac{1}{2}(Z_k + Z_k^{-1}),
\end{equation} 
which corresponds to the choice $h(Z) = \frac{1}{2}(I + Z^{-1})$.  It is known that this iteration converges quadratically to $\mathrm{sign}(Z_0)$ for any $Z_0$ having no pure imaginary eigenvalues~\cite[Theorem 5.6]{higham2008functions}.  
Since~(\ref{Z0}) and~(\ref{Z0b}) have eigenvalues equal to plus or minus the eigenvalues of $H$, all of which are positive real numbers, we obtain the following corollary.  In it, we restrict the discussion to square matrices, since this leads to a particularly simple iterative scheme.
\begin{corollary} \label{cor:Newton}
Let $A \in \mathbb{C}^{n \times n}$ be a nonsingular matrix having polar decomposition $A=UH$, where $U \in \mathbb{C}^{n \times n}$ is unitary and $H \in \mathbb{C}^{n \times n}$ is Hermitian positive definite.  Let $E \in \mathbb{C}^{n \times n}$.  Then the iteration
\begin{alignat}{3}
X_{k+1} &= \frac{1}{2}(X_k + X_k^{-*}), &\quad& X_0 = A, \label{XupdateNewton} \\
E_{k+1} &= \frac{1}{2}(E_k - X_k^{-*} E_k^* X_k^{-*}), &\quad& E_0 = E, \label{EupdateNewton}
\end{alignat}
produces iterates $X_k$ and $E_k$ that converge to $\mathcal{P}(A)=U$ and $L_{\mathcal{P}}(A,E)$, respectively, as $k \rightarrow \infty$.
\end{corollary}
\paragraph{Remark} 
If $A$ is rectangular, then the iteration obtained from~(\ref{signNewton}) reads
\begin{alignat}{3}
X_{k+1} &= \frac{1}{2}X_k(I + (X_k^* X_k)^{-1}), &\quad& X_0 = A, \label{XupdateNewtonrect}\\
E_{k+1} &= \frac{1}{2}\Big[E_k(I + (X_k^* X_k)^{-1}) &&\label{EupdateNewtonrect}\\ &\hspace{0.3in}- X_k  (X_k^* X_k)^{-1} (E_k^* X_k + X_k^* E_k)  (X_k^* X_k)^{-1}\Big], &\quad& E_0 = E.  \nonumber
\end{alignat}
This scheme simplifies to~(\ref{XupdateNewton}-\ref{EupdateNewton}) when $A$ is square.

A second corollary of Theorem~\ref{thm:dpolariter} is obtained by considering the Newton-Schulz iteration
\begin{equation} \label{NewtonSchulz}
Z_{k+1} = \frac{1}{2}Z_k(3I-Z_k^2),
\end{equation}
which corresponds to the choice $h(Z) = 3I-Z$.  It is known that this iteration converges to $\mathrm{sign}(Z_0)$ provided that (i)~$Z_0$ has no pure imaginary eigenvalues and (ii)~the eigenvalues of $I-Z_0^2$ all have magnitude strictly less than one~\cite[Theorem 5.2]{kenney1991rational}.  Note that~\cite[Theorem 5.2]{kenney1991rational} replaces the latter condition with $\|I-Z_0^2\|<1$, but it is evident from their proof that this condition can be relaxed to what we have written here.  Since the eigenvalues of
\[
\begin{pmatrix} I & 0 \\ 0 & I \end{pmatrix} - \begin{pmatrix} H & \Omega \\ 0 & -H \end{pmatrix}^2 = \begin{pmatrix} I - H^2 & \Omega H - H \Omega \\ 0 & I- H^2 \end{pmatrix}
\]
and
\[
\begin{pmatrix} I & 0 \\ 0 & I \end{pmatrix} - \begin{pmatrix} H & S \\ 0 & H \end{pmatrix}^2 = \begin{pmatrix} I - H^2 & -H S - S H \\ 0 & I- H^2 \end{pmatrix}
\]
coincide with those of $I-H^2 = I-A^* A$, we obtain the following corollary.
\begin{corollary} \label{cor:NewtonSchulz}
Let $A \in \mathbb{C}^{m \times n}$ ($m \ge n$) be a full-rank matrix having polar decomposition $A=UH$, where $U \in \mathbb{C}^{m \times n}$ has orthonormal columns and $H \in \mathbb{C}^{n \times n}$ is Hermitian positive definite.  Let $E \in \mathbb{C}^{m \times n}$.  If all of the singular values of $A$ lie in the interval $(0,\sqrt{2})$, then the iteration
\begin{alignat}{3}
X_{k+1} &= \frac{1}{2}X_k(3I - X_k^* X_k), &\quad& X_0 = A, \label{XupdateNewtonSchulz} \\
E_{k+1} &= \frac{1}{2} E_k(3I - X_k^* X_k) - \frac{1}{2} X_k (E_k^* X_k + X_k^* E_k) , &\quad& E_0 = E, \label{EupdateNewtonSchulz}
\end{alignat}
produces iterates $X_k$ and $E_k$ that converge to $\mathcal{P}(A)=U$ and $L_{\mathcal{P}}(A,E)$, respectively, as $k \rightarrow \infty$.
\end{corollary}

\paragraph{Remark} A more direct analysis of~(\ref{XupdateNewtonSchulz}), without appealing to its relationship to a matrix sign function iteration, shows that $X_k \rightarrow U$ under the less stringent requirement that all of the singular values of $A$ lie in the interval $(0,\sqrt{3})$~\cite[Problem 8.20]{higham2008functions}.  Our numerical experiments suggest that the coupled iteration~(\ref{XupdateNewtonSchulz}-\ref{EupdateNewtonSchulz}) enjoys convergence under the same condition, but Theorem~\ref{thm:dpolariter} alone appears inadequate to conclude such a claim.

Other corollaries to Theorem~\ref{thm:dpolariter} can be derived in a similar fashion.  For instance, iterative schemes based on Pad\'{e} approximations of $\mathrm{sign}(Z) = Z(I-(I-Z^2))^{-1/2}$ (of which~(\ref{NewtonSchulz}) is a special case) can be used; see~\cite[Chapter 5.4]{higham2008functions} for further details.

\section{Proofs} \label{sec:proofs}

In this section, we present proofs of Theorems~\ref{thm:dpolariter} and~\ref{thm:signblockmat}.  Our presentation is divided into two parts.  First, in Section~\ref{sec:identities}, we derive a few identities involving the Fr\'{e}chet derivative of the polar decomposition, proving Theorem~\ref{thm:signblockmat} in the process.  Then, in Section~\ref{sec:thmproof}, we use the aforementioned identites to prove convergence of the iteration~(\ref{Xupdate2}-\ref{Eupdate2}), thereby proving Theorem~\ref{thm:dpolariter}.

\subsection{Identities Involving the Fr\'{e}chet Derivative of the Polar Decomposition} \label{sec:identities}

This section studies the Fr\'{e}chet derivative of the polar decomposition and its relationship to the matrix sign function, culminating in a proof of Theorem~\ref{thm:signblockmat}.
A couple of main observations will be made.
First, as will be seen in Lemma~\ref{lemma:dpolar}, the task of evaluating $L_{\mathcal{P}}(A,E)$ can essentially be reduced to the case in which $A$ is Hermitian positive definite and $E$ is skew-Hermitian.  This is relatively simple to show when $A$ is square, but the rectangular case turns out to be more subtle, requiring some that some attention be paid to the relationship between the column space of $A$ and that of $E$.  This observation will be followed with a proof of Theorem~\ref{thm:signblockmat}, which reveals that the value of $U^* L_{\mathcal{P}}(A,E)$ can be read off of the $(1,2)$-block  of the matrix sign function applied to a certain block matrix.

Before studying the derivatives of $\mathcal{P}$ in detail, it is worth pointing out that $\mathcal{P}$ is a smooth map from the set of full-rank $m \times n$ ($m \ge n$) matrices to the set of  $m \times n$ matrices with orthonormal columns.  This follows from two facts: (1) the latter set of matrices constitutes a smooth, compact manifold, the Stiefel manifold $V_n(\mathbb{C}^m) = \{U \in \mathbb{C}^{m \times n} \mid U^*U = I\}$, and (2) the map $\mathcal{P}$ coincides with the closest point projection onto $V_n(\mathbb{C}^n)$.  That is, in the Frobenius norm $\|\cdot\|_F$,
\[
\mathcal{P}(A) = \argmin_{U \in V_n(\mathbb{C}^m)} \|A-U\|_F
\]
for any full-rank $A \in \mathbb{C}^{m \times n}$~\cite[Theorem 8.4]{higham2008functions}.  It is a classical result from differential geometry that the closest point projection onto a smooth, compact manifold embedded in Euclidean space is a smooth map~\cite{foote1984regularity}.
In particular, $\mathcal{P}$ is Fr\'{e}chet differentiable at any full-rank $A \in \mathbb{C}^{m \times n}$.  (For a different justification of this fact, see~\cite[Section 2.3(c)]{dieci1999smooth}.)

We now turn our attention to the differentiation of $\mathcal{P}$. We begin by recording a useful formula for the Fr\'{e}chet derivative of a function of the form $g(X) = X h(X^* X)$.  Along the way, we make some observations concerning the column space $\mathcal{R}(A)$ of a matrix $A \in \mathbb{C}^{m \times n}$ and the column space $\mathcal{R}(L_g(A,E))$ of the Fr\'{e}chet derivative $L_g(A,E)$ of $g$ at $A$ in a direction $E \in \mathbb{C}^{m \times n}$.  We denote by $\mathcal{N}(A^*)$ the null space of $A^*$; equivalently, $\mathcal{N}(A^*)$ is the orthogonal complement to $\mathcal{R}(A)$ in $\mathbb{C}^m$.

\begin{lemma} \label{lemma:Lg}
Let $A \in \mathbb{C}^{m \times n}$ ($m \ge n$), let $h : \mathbb{C}^{n \times n} \rightarrow \mathbb{C}^{n \times n}$ be Fr\'{e}chet differentiable at $A^* A$, and define $g(X) = X h(X^* X)$.  Then for any $E \in \mathbb{C}^{m \times n}$,
\begin{equation} \label{Lg}
L_g(A,E) = E h(A^* A) + A L_h(A^* A, A^*E + E^* A).
\end{equation}
In particular, if  $\mathcal{R}(E) \subseteq \mathcal{R}(A)$, then $\mathcal{R}(L_g(A,E)) \subseteq \mathcal{R}(A)$.  On the other hand, if $\mathcal{R}(E) \subseteq \mathcal{N}(A^*)$, then
\begin{equation} \label{Lgperp}
L_g(A,E) = E h(A^* A),
\end{equation}
and hence $\mathcal{R}(L_g(A,E)) \subseteq \mathcal{N}(A^*)$.
\end{lemma}
\begin{proof}
The formula~(\ref{Lg}) is a consequence of the product rule and the chain rule~\cite[Theorems 3.3 \& 3.4]{higham2008functions}.  The implication $\mathcal{R}(E) \subseteq \mathcal{R}(A) \implies \mathcal{R}(L_g(A,E)) \subseteq \mathcal{R}(A)$ is immediate since the columns of $L_g(A,E)$ are linear combinations of the columns of $A$ and $E$.  Equation~(\ref{Lgperp}) follows from the fact that $A^*E+E^*A = 0$ whenever $\mathcal{R}(E) \subseteq \mathcal{N}(A^*)$.
\end{proof}

The preceding lemma has several important consequences.  The first of these is an application of Lemma~\ref{lemma:Lg} to the function $g(X) = \mathcal{P}(X)$, which has the requisite functional form in view of~(\ref{signpolardef}).

\begin{lemma} \label{lemma:Lfo}
Let $A \in \mathbb{C}^{m \times n}$ ($m \ge n$) be a full-rank matrix having polar decomposition $A=UH$, where $U = \mathcal{P}(A) \in \mathbb{C}^{m \times n}$ has orthonormal columns and $H \in \mathbb{C}^{n \times n}$ is Hermitian positive definite.  Let $E \in \mathbb{C}^{m \times n}$, and write
\[
E = E^\parallel + E^\perp, \quad E^\parallel = UU^* E, \quad  E^\perp = (I-UU^*) E.
\]
Then
\begin{equation} \label{dpolarpar}
U U^* L_{\mathcal{P}}(A,E^\parallel) = L_{\mathcal{P}}(A,E^\parallel)
\end{equation}
and
\begin{equation} \label{dpolarperp}
L_{\mathcal{P}}(A,E^\perp) = E^\perp H^{-1}.
\end{equation}
\end{lemma}
\begin{proof}
Apply Lemma~\ref{lemma:Lg} with the choice $h(X) = X^{-1/2}$, so that $g(X) = X (X^*X)^{-1/2} = \mathcal{P}(X)$.  Equation~(\ref{dpolarpar}) is a restatement of the fact that $\mathcal{R}(L_{\mathcal{P}}(A,E^\parallel)) \subseteq \mathcal{R}(A) = \mathcal{R}(U)$, while~(\ref{dpolarperp}) follows from~(\ref{Lgperp}) together with the identity $H = (A^*A)^{1/2}$.
\end{proof}

We will now show, with the help of Lemma~\ref{lemma:Lfo}, that the task of evaluating $L_{\mathcal{P}}(A,E)$ can essentially be reduced to the case in which $A$ is Hermitian positive definite and $E$ is skew-Hermitian.  

\begin{lemma} \label{lemma:dpolar}
Let $A \in \mathbb{C}^{m \times n}$ ($m \ge n$) be a full-rank matrix having polar decomposition $A=UH$, where $U \in \mathbb{C}^{m \times n}$ has orthonormal columns and $H \in \mathbb{C}^{n \times n}$ is Hermitian positive definite.  Then for any $E \in \mathbb{C}^{m \times n}$,
\begin{align}
\mathrm{skew}(U^* L_{\mathcal{P}}(A,E)) &= L_{\mathcal{P}}(H,\Omega), \label{skewLfo} \\
\mathrm{sym}(U^* L_{\mathcal{P}}(A,E)) &= L_{\mathcal{P}}(H,S) = 0, \label{symLfo}
\end{align}
where $\Omega = \mathrm{skew}(U^* E)$ and $S = \mathrm{sym}(U^* E)$.  Hence,
\begin{equation} \label{dpolarrelation}
U^* L_{\mathcal{P}}(A,E) = L_{\mathcal{P}}(H,\Omega).
\end{equation}
\end{lemma}
\begin{proof}

Decompose $E$ as
\[
E = E^\parallel + E^\perp, \quad E^\parallel = UU^* E, \quad  E^\perp = (I-UU^*) E.
\]
The linearity of the Fr\'{e}chet derivative implies that
\begin{align}
L_{\mathcal{P}}(A,E) 
&= L_{\mathcal{P}}(A,E^\parallel) +  L_{\mathcal{P}}(A,E^\perp). \nonumber
\end{align}
The formula~(\ref{dpolarperp}) and the identities $A=UH$ and $UU^* E^\parallel = E^\parallel$ then give
\[
L_{\mathcal{P}}(A,E) = L_{\mathcal{P}}(UH,UU^*E^\parallel) + E^\perp H^{-1}.
\]
Now note that the map $\mathcal{P}$ clearly satisfies $\mathcal{P}(VB)=V \mathcal{P}(B)$ for any nonsingular $B \in \mathbb{C}^{n \times n}$ and any $V \in \mathbb{C}^{m \times n}$ ($m \ge n$) with orthonormal columns.  From this it follows that for any such $V$ and $B$, and any $F \in \mathbb{C}^{n \times n}$,
\begin{equation} \label{dpolarsymmetry}
L_{\mathcal{P}}(VB,VF) =  V L_{\mathcal{P}}(B, F).
\end{equation}
Applying this identity to the case in which $B = H$, $V=U$, and $F = U^*E^\parallel$, we obtain
\begin{align*}
L_{\mathcal{P}}(UH,UU^* E^\parallel) 
&= U L_{\mathcal{P}}(H,U^* E^\parallel) \\
&= U L_{\mathcal{P}}(H,U^* E),
\end{align*}
where the second line follows from the fact that $U^* E^\perp = 0$.
Thus,
\[
L_{\mathcal{P}}(A,E) = U L_{\mathcal{P}}(H,U^* E) + E^\perp H^{-1}.
\]
Multiplying from the left by $U^*$ gives
\[
U^* L_{\mathcal{P}}(A,E) = L_{\mathcal{P}}(H,U^* E).
\]
since $U^*U = I$ and $U^* E^\perp = 0$.  Equivalently, in terms of $\Omega = \mathrm{skew}(U^* E)$ and $S = \mathrm{sym}(U^* E)$, 
\[
U^* L_{\mathcal{P}}(A,E) = L_{\mathcal{P}}(H,\Omega) + L_{\mathcal{P}}(H,S)
\] 
The proof will be complete if we can show that $L_{\mathcal{P}}(H, \Omega)$ is skew-Hermitian and
\begin{equation} \label{dpolarsym}
L_{\mathcal{P}}(H,S) = 0.
\end{equation}
In fact,~(\ref{dpolarsym}) holds for any Hermitian matrix $S$ since, for all sufficiently small $\varepsilon$, $H+\varepsilon S$ is Hermitian positive definite, showing that $\mathcal{P}(H+\varepsilon S)=I$.  The skew-Hermiticity of $L_{\mathcal{P}}(H, \Omega)$ follows from differentiating the identity
\[
\mathcal{P}(H+\varepsilon \Omega)^* \mathcal{P}(H+\varepsilon \Omega) = I
\]
with respect to $\varepsilon$ and using the fact that $\mathcal{P}(H)=I$.
\end{proof}

Another consequence of Lemma~\ref{lemma:Lg} is the following identity that relates the Fr\'{e}chet derivative of the polar decomposition of a Hermitian positive definite matrix to the matrix sign function applied to a certain block matrix.

\begin{lemma} \label{lemma:signsym}
Let $H \in \mathbb{R}^{n \times n}$ be Hermitian positive definite, and let $\Omega \in \mathbb{R}^{n \times n}$ be skew-Hermitian.  Then
\begin{align}
\mathrm{sign}\begin{pmatrix} H & \Omega \\ 0 & -H \end{pmatrix} &= \begin{pmatrix} I & L_{\mathcal{P}}(H,\Omega) \\ 0 & -I \end{pmatrix}, \label{signHW} \\
\mathrm{sign} \begin{pmatrix} H & S \\ 0 & H \end{pmatrix} &= \begin{pmatrix} I & L_{\mathcal{P}}(H,S) \\ 0 & I \end{pmatrix} = \begin{pmatrix} I & 0 \\ 0 & I \end{pmatrix}. \label{signHS}
\end{align}
\end{lemma}
\begin{proof}
By definition,
\begin{align*}
\mathrm{sign}\begin{pmatrix} H & \Omega \\ 0 & -H \end{pmatrix}
&= \begin{pmatrix} H & \Omega \\ 0 & -H \end{pmatrix} \begin{pmatrix} H^2 & H\Omega-\Omega H \\ 0 & H^2 \\ \end{pmatrix}^{-1/2} \\
&= \begin{pmatrix} H & \Omega \\ 0 & -H \end{pmatrix} \begin{pmatrix} H^2 & H\Omega+\Omega^* H \\ 0 & H^2 \\ \end{pmatrix}^{-1/2}.
\end{align*}
Now apply~(\ref{mathias}) to the primary matrix function $f(X)=X^{-1/2}$ to obtain
\begin{align*}
\begin{pmatrix} H^2 & H\Omega+\Omega^* H \\ 0 & H^2 \\ \end{pmatrix}^{-1/2} &= \begin{pmatrix} H^{-1} & L_{x^{-1/2}}(H^2,H\Omega+\Omega^* H) \\ 0 & H^{-1} \\ \end{pmatrix},
\end{align*}
where the identity $(H^2)^{-1/2}=H^{-1}$ follows from the positive-definiteness of $H$.  Thus,
\begin{align*}
\mathrm{sign}\begin{pmatrix} H & \Omega \\ 0 & -H \end{pmatrix} 
&= \begin{pmatrix} H & \Omega \\ 0 & -H \end{pmatrix} \begin{pmatrix} H^{-1} & L_{x^{-1/2}}(H^2,H\Omega+\Omega^* H) \\ 0 & H^{-1} \\ \end{pmatrix} \\
&= \begin{pmatrix} I & H L_{x^{-1/2}}(H^2,H\Omega+\Omega^* H) + \Omega H^{-1} \\ 0 & -I \end{pmatrix}.
\end{align*}
The identity~(\ref{signHW}) follows upon observing that, by~(\ref{Lg}),
\begin{align*}
L_{\mathcal{P}}(H,\Omega) &= \Omega H^{-1} + H L_{x^{-1/2}}(H^2,H\Omega+\Omega^* H).
\end{align*}
The proof of~(\ref{signHS}) is simpler, since, by~(\ref{mathias}) and~(\ref{symLfo}),
\begin{displaymath}
\mathrm{sign} \begin{pmatrix} H & S \\ 0 & H \end{pmatrix} = \begin{pmatrix} I & L_{\mathrm{sign}}(H,S) \\ 0 & I \end{pmatrix} = \begin{pmatrix} I & 0 \\ 0 & I \end{pmatrix}. 
\end{displaymath}
\end{proof}

We remark that an alternative proof of~(\ref{signHW}) exists.  It is based on the observation that $L_{\mathcal{P}}(H,\Omega)$ is the solution of a Lyapunov equation which can be solved by reading off the $(1,2)$-block of $\mathrm{sign}\begin{pmatrix} H & \Omega \\ 0 & -H \end{pmatrix}$.  For details, see Section~\ref{sec:comparison}.

Combining Lemma~\ref{lemma:signsym} with Lemma~\ref{lemma:dpolar} proves Theorem~\ref{thm:signblockmat}.

\subsection{Convergence of the Iteration} \label{sec:thmproof}

We now focus our efforts on proving convergence of the iteration~(\ref{Xupdate2}-\ref{Eupdate2}), thereby proving Theorem~\ref{thm:dpolariter}.  The cornerstone of the proof is Lemma~\ref{lemma:relation}, where a relationship is established between certain blocks of the matrices $Z_k$ defined by the matrix sign function $Z_{k+1} = Z_k h(Z_k)^2$ and the matrices $X_k$ and $E_k$ defined by the iteration~(\ref{Xupdate2}-\ref{Eupdate2}).  Once this has been shown, convergence of the iteration~(\ref{Xupdate2}-\ref{Eupdate2}) will follow from the convergence of $Z_k$ to $\mathrm{sign}(Z_0)$, together with the knowledge (from Theorem~\ref{thm:signblockmat}) that the Fr\'{e}chet derivative of the polar decomposition is related to the $(1,2)$-block of $\mathrm{sign}(Z_0)$ for certain values of $Z_0$.

We begin by examining the block structure of the iterates $Z_k$.

\begin{lemma} \label{lemma:YkOmegak}
The iterates $Z_k$ produced by the iteration $Z_{k+1} = Z_k h(Z_k^2)$ with initial condition~(\ref{Z0}) have the form
\[
Z_k = \begin{pmatrix} H_k & \Omega_k \\ 0 & -H_k \end{pmatrix},
\]
where $H_k$ is Hermitian and $\Omega_k$ is skew-Hermitian.  
\end{lemma}
\begin{proof}
Assume the statement is true at iteration $k$.  Then by~(\ref{mathias}),
\begin{align}
Z_{k+1} 
&= \begin{pmatrix} H_k & \Omega_k \\ 0 & -H_k \end{pmatrix} h \begin{pmatrix} H_k^2 & H_k \Omega_k - \Omega_k H_k \\ 0 & H_k^2 \end{pmatrix} \nonumber \\
&= \begin{pmatrix} H_k & \Omega_k \\ 0 & -H_k \end{pmatrix} \begin{pmatrix} h(H_k^2) & L_h( H_k^2, H_k \Omega_k - \Omega_k H_k) \\ 0 & h(H_k^2) \end{pmatrix} \nonumber \\
&= \begin{pmatrix} H_k h(H_k^2) & H_k L_h( H_k^2, H_k \Omega_k - \Omega_k H_k) +\Omega_k h(H_k^2) \\ 0 & -H_k h(H_k^2) \end{pmatrix}. \label{Zkp1}
\end{align}
By the remark following Theorem~\ref{thm:dpolariter}, $H_k h(H_k^2) = \left[ H_k h(H_k^2) \right]^*$,
showing that $H_{k+1} = H_k h(H_k^2)$ is Hermitian.  On the other hand, the fact that $h$ is a primary matrix function implies that $Z_k$ commutes with $h(Z_k^2)$, so, by a calculation similar to that above, we also have
\begin{align}
Z_{k+1}
&= \begin{pmatrix} h(H_k^2) H_k  & h(H_k^2) \Omega_k - L_h( H_k^2, H_k \Omega_k - \Omega_k H_k) H_k  \\ 0 & -h(H_k^2) H_k \end{pmatrix}. \label{Zkp1b}
\end{align}
Denote $C_k = H_k \Omega_k - \Omega_k H_k$.  Since $H_k$ is Hermitian and $\Omega_k$ is skew-Hermitian, $C_k$ is Hermitian.  Hence, since $h(Z^*)=h(Z)^*$ for every $Z$,
\begin{align*}
L_h(H_k^2,C_k)^* 
&= L_h((H_k^2)^*,C_k^*) \\
&= L_h(H_k^2,C_k).
\end{align*}
Comparing the $(1,2)$ blocks of~(\ref{Zkp1}) and~(\ref{Zkp1b}) then shows that
\begin{align}
0 
&= H_k L_h(H_k^2,C_k) + \Omega_k h(H_k^2) - h(H_k^2) \Omega_k + L_h(H_k^2,C_k) H_k  \nonumber \\
&= H_k L_h(H_k^2,C_k) + \Omega_k h(H_k^2) + h(H_k^2)^* \Omega_k^* + L_h(H_k^2,C_k)^* H_k^* \nonumber \\
&= \Omega_{k+1} + \Omega_{k+1}^*.
\end{align}
It follows that $\Omega_k=-\Omega_k^*$ for every $k$.
\end{proof}

The proof above also reveals a recursion satisfied by $H_k$ and $\Omega_k$, namely,
\begin{align}
H_{k+1} &= H_k h(H_k^2) \label{Yupdate} \\
\Omega_{k+1} &= \Omega_k h(H_k^2) + H_k L_h(H_k^2, H_k \Omega_k - \Omega_k H_k). \label{Omegaupdate}
\end{align}

Next, we examine the block structure of the iterates $Z_k$ with initial condition~(\ref{Z0b}).

\begin{lemma} \label{lemma:YkSk}
The iterates $Z_k$ produced by the iteration $Z_{k+1} = Z_k h(Z_k^2)$ with initial condition~(\ref{Z0b}) have the form
\[
Z_k = \begin{pmatrix} H_k & S_k \\ 0 & H_k \end{pmatrix},
\]
where $H_k$ is the same Hermitian matrix as in Lemma~\ref{lemma:YkOmegak} and $S_k$ is Hermitian.  
\end{lemma}
\begin{proof}
We omit the proof, which is very similar to the proof of Lemma~\ref{lemma:YkOmegak}.
\end{proof}

In analogy with~(\ref{Omegaupdate}), the iterates $S_k$ satisfy the recursion 
\begin{equation}
 S_{k+1} = S_k h(H_k^2) + H_k L_h(H_k^2, S_k H_k + H_k S_k). \label{Supdate}
\end{equation}

We now relate the matrices $H_k$, $\Omega_k$, and $S_k$ defined in the preceding pair of lemmas to the matrices $X_k$ and $E_k$ defined by the coupled iteration~(\ref{Xupdate2}-\ref{Eupdate2}).

\begin{lemma} \label{lemma:relation}
The iterates $H_k$, $\Omega_k$, and $S_k$ are related to $X_k$ and $E_k$ via
\begin{align}
U H_k &= X_k, \label{relationXY} \\
\Omega_k &= \mathrm{skew}(U^* E_k), \label{relationOmegaE} \\
S_k &= \mathrm{sym}(U^* E_k). \label{relationSE}
\end{align}
\end{lemma}
\begin{proof}
The first of these equalities follows easily by induction, for if it holds at iteration $k$, then
\begin{align*}
X_{k+1} 
&= g(X_k) \\
&= X_k h(X_k^* X_k) \\
&= U H_k h(H_k^* U^* U H_k) \\
&= U H_k h(H_k^2) \\
&= U H_{k+1}.
\end{align*}
Furthermore, $X_0 = A = UH = UH_0$, which proves~(\ref{relationXY}).  
To prove~(\ref{relationOmegaE}) and~(\ref{relationSE}), we will show that if $\Omega_k = \mathrm{skew}(U^* E_k)$ and $S_k = \mathrm{sym}(U^* E_k)$ for a given $k$, and if $E_{k+1}$, $\Omega_{k+1}$, and $S_{k+1}$ are given by~(\ref{Eupdate2}),~(\ref{Omegaupdate}), and~(\ref{Supdate}), respectively, then $\Omega_{k+1} = \mathrm{skew}(U^* E_{k+1})$ and $S_{k+1} = \mathrm{sym}(U^* E_{k+1})$.  
Recalling~(\ref{Lg}), we have
\begin{align*}
U^* E_{k+1}
\hspace{-2em}&\hspace{2em}= U^* L_g(X_k,E_k) \\
&= U^* E_k h(X_k^* X_k) + U^* X_k L_h(X_k^* X_k, E_k^* X_k + X_k^* E_k) \\
&= U^* E_k h(H_k^2) + H_k L_h(H_k^2, E_k^* U H_k + H_k U^* E_k) \\
&= \Omega_k h(H_k^2) + H_k L_h(H_k^2, \Omega_k^* H_k + H_k \Omega_k) + S_k h(H_k^2) + H_k L_h(H_k^2, S_k H_k + H_k S_k) \\
&= \Omega_{k+1} + S_{k+1},
\end{align*}
where we have used~(\ref{Omegaupdate}),~(\ref{Supdate}), and the decomposition $U^* E_k = \Omega_k+S_k$.
By Lemmas~\ref{lemma:YkOmegak} and~\ref{lemma:YkSk}, $\Omega_{k+1}$ is skew-Hermitian and $S_{k+1}$ is Hermitian, proving~(\ref{relationOmegaE}) and~(\ref{relationSE}).
\end{proof}

The proof of Theorem~\ref{thm:dpolariter} is now almost complete, since by Lemma~\ref{lemma:YkOmegak} and Theorem~\ref{thm:signblockmat},
\[
\begin{pmatrix} H_k & \Omega_k \\ 0 & -H_k \end{pmatrix} \rightarrow \mathrm{sign} \begin{pmatrix} H & \Omega \\ 0 & -H \end{pmatrix} = \begin{pmatrix} I & U^* L_{\mathcal{P}}(A,E) \\ 0 & -I \end{pmatrix}
\]
as $k \rightarrow \infty$.  Likewise, by~(\ref{signHS}) and Lemma~\ref{lemma:YkSk},
\[
\begin{pmatrix} H_k & S_k \\ 0 & H_k \end{pmatrix} \rightarrow \mathrm{sign} \begin{pmatrix} H & S \\ 0 & H \end{pmatrix} = \begin{pmatrix} I & 0 \\ 0 & I \end{pmatrix},
\]
as $k \rightarrow \infty$.  
These observations, together with~(\ref{relationXY}-\ref{relationSE}), show that
\begin{align*}
X_k &\rightarrow U, \\
\mathrm{skew}(U^* E_k) &\rightarrow U^* L_{\mathcal{P}}(A,E), \\
\mathrm{sym}(U^* E_k) &\rightarrow 0
\end{align*}
as $k \rightarrow \infty$.  In other words,
\begin{align}
X_k &\rightarrow \mathcal{P}(A), \label{Xklimit} \\
U^* E_k &\rightarrow U^* L_{\mathcal{P}}(A,E) \label{Eklimit}
\end{align}
as $k \rightarrow \infty$.  The latter limit implies that $E_k \rightarrow L_{\mathcal{P}}(A,E)$ when $U$ is square, but not when $U$ is rectangular.  To handle the rectangular case, consider the decompositions
\begin{alignat*}{3}
E_k &= E_k^\parallel + E_k^\perp, \quad E_k^\parallel &= UU^* E_k, \quad E_k^\perp &= (I-UU^*) E_k,  \\
E &= E^\parallel + E^\perp, \quad E^\parallel &= UU^* E, \quad E^\perp &= (I-UU^*) E.
\end{alignat*}
By Lemma~\ref{lemma:Lfo} and the linearity of the Fr\'{e}chet derivative, the statement~(\ref{Eklimit}) is equivalent to the statement that
\begin{align*}
U^* E_k^\parallel 
&\rightarrow U^*L_{\mathcal{P}}(A,E^\parallel) + U^*L_{\mathcal{P}}(A,E^\perp) \\
&= U^*L_{\mathcal{P}}(A,E^\parallel).
\end{align*}
Multiplying from the left by $U$ and recalling that $UU^*E_k^\parallel = E_k^\parallel$ and $UU^*L_{\mathcal{P}}(A,E^\parallel)=L_{\mathcal{P}}(A,E^\parallel)$ (by~(\ref{dpolarpar})), we conclude that
\begin{equation}
E_k^\parallel \rightarrow L_{\mathcal{P}}(A,E^\parallel).
\end{equation}

The proof will of Theorem~\ref{thm:dpolariter} be complete if we can show that 
\begin{equation} \label{Ekperplimit}
E_k^\perp \rightarrow L_{\mathcal{P}}(A,E^\perp).
\end{equation}
This is carried out in the following lemma.

\begin{lemma}
As $k \rightarrow \infty$, $E_k^\perp \rightarrow L_{\mathcal{P}}(A,E^\perp)$.
\end{lemma}
\begin{proof}
By~(\ref{dpolarperp}), it suffices to show that
\[
E_k^\perp \rightarrow E^\perp H^{-1}.
\]
Using Lemma~\ref{lemma:Lg}, it is straightforward to see that
$E_k^\parallel$ and $E_k^\perp$ satisfy independent recursions of the form
\begin{align*}
E_{k+1}^\parallel &= L_g(X_k, E_k^\parallel), \\
E_{k+1}^\perp &= L_g(X_k, E_k^\perp). 
\end{align*}
Now since $\mathcal{R}(E_k^\perp)$ is orthogonal to $\mathcal{R}(U) \supseteq \mathcal{R}(UH_k) = \mathcal{R}(X_k)$, it follows from~(\ref{Lgperp}) that
\[
L_g(X_k, E_k^\perp) = E_k^\perp h(X_k^* X_k),
\]
so 
\begin{align*}
E_{k+1}^\perp 
&= E_k^\perp h(X_k^* X_k). 
\end{align*}
If we introduce the matrix $B_k \in \mathbb{C}^{n \times n}$ defined by the recursion
\[
B_{k+1} = B_k h(X_k^* X_k), \quad B_0 = I,
\]
then an inductive argument shows that
\[
E_k^\perp = E^\perp B_k.
\]
We claim that $B_k \rightarrow H^{-1}$ as $k \rightarrow \infty$.  To see this, observe that~(\ref{Xupdate2}) implies that
\[
X_k = X_0 B_k = A B_k.
\]
Since $X_k \rightarrow U$ as $k \rightarrow \infty$, we conclude that
\[
I = U^*U = U^* \lim_{k \rightarrow \infty} X_k =  U^* A \lim_{k \rightarrow \infty} B_k = H \lim_{k \rightarrow \infty} B_k.
\]
It follows that $E_k^\perp = E^\perp B_k  \rightarrow E^\perp H^{-1}$ as $k \rightarrow \infty$.
\end{proof}

\section{Practical Considerations} \label{sec:practical}

This section discusses several practical considerations concerning the iterative schemes detailed in Theorem~\ref{thm:dpolariter}.

\subsection{Scaling}

Scaling the iterates $X_k$ in the Newton iteration~(\ref{XupdateNewton}) often reduces the number of iterations required to achieve convergence~\cite[Chapter 8.6]{higham2008functions}.  If this strategy is generalized to the coupled iteration~(\ref{XupdateNewton}-\ref{EupdateNewton}), then the resulting iteration reads
\begin{alignat}{3}
X_{k+1} &= \frac{1}{2}(\mu_k X_k + \mu_k^{-1} X_k^{-*}), &\quad& X_0 = A, \label{XupdateNewtonscaled} \\
E_{k+1} &= \frac{1}{2}(\mu_k E_k - \mu_k^{-1} X_k^{-*} E_k^* X_k^{-*}), &\quad& E_0 = E, \label{EupdateNewtonscaled}
\end{alignat}
where $\mu_k > 0$ is a scaling factor chosen heuristically. Practical choices for $\mu_k$ include~\cite{higham2008functions}
\begin{equation} \label{scale1inf}
\mu_k = \left( \frac{\|X_k^{-1}\|_1 \|X_k^{-1}\|_{\infty}}{\|X_k\|_1 \|X_k\|_{\infty}} \right)^{1/4}
\end{equation}
and
\begin{equation} \label{scaleF}
\mu_k = \left( \frac{\|X_k^{-1}\|_F }{\|X_k\|_F } \right)^{1/2},
\end{equation}
where $\|\cdot\|_1$, $\|\cdot\|_\infty$, and $\|\cdot\|_F$ denote the matrix $1$-, $\infty$- and Frobenius norms, respectively.

More generally, scaling can be applied to other iterative schemes of the form~(\ref{Xupdate2}-\ref{Eupdate2}), leading to iterative schemes of the form
\begin{alignat}{3}
X_{k+1} &= g(\mu_k X_k), &\quad& X_0 = A, \label{Xupdatescaled} \\
E_{k+1} &= L_g(\mu_k X_k, \mu_k E_k), &\quad& E_0 = E. \label{Eupdatescaled}
\end{alignat}
Note that if $A$ is rectangular, then~(\ref{scale1inf}) and~(\ref{scaleF}) are inapplicable.  We have found 
\begin{equation} \label{scale1infrect}
\mu_k = \left( \frac{\|(X_k^*X_k)^{-1}\|_1 \|(X_k^*X_k)^{-1}\|_{\infty}}{\|X_k^*X_k\|_1 \|X_k^*X_k\|_{\infty}} \right)^{1/8}
\end{equation}
and
\begin{equation} \label{scaleFrect}
\mu_k = \left( \frac{\|(X_k^*X_k)^{-1}\|_F }{\|X_k^*X_k\|_F } \right)^{1/4}
\end{equation}
to be effective alternatives to~(\ref{scale1inf}) and~(\ref{scaleF}) in our numerical experiments with rectangular $A$.

\subsection{Termination Criteria} \label{sec:termination}

Determining when to terminate the iteration~(\ref{Xupdate2}-\ref{Eupdate2}) is a delicate task.  Termination criteria for~(\ref{Xupdate2}) by itself are, of course, well-studied, but the accuracy of $E_k$ should be taken into account when choosing termination criteria for the coupled iteration~(\ref{Xupdate2}-\ref{Eupdate2}).  

One possibility is to appeal to the relationship between $X_k$ and $E_k$ and the sign function iterates $Z_k$ referenced in the statement of Theorem~\ref{thm:dpolariter}.  Convergence of the sign function iterates to $\mathrm{sign}(Z_0) = \lim_{k\rightarrow \infty} Z_k$ can be readily verified with the aid of the inequality
\[
\frac{\|Z_k^2 - I\|}{\|\mathrm{sign}(Z_0)\|(\|Z_k\|+\|\mathrm{sign}(Z_0)\|} \le \frac{\|Z_k-\mathrm{sign}(Z_0)\|}{\|\mathrm{sign}(Z_0)\|} \le \|Z_k^2 - I\|,
\]
which holds in any submultiplicative matrix norm, so long as $\|\mathrm{sign}(Z_0)(Z_k\linebreak-\mathrm{sign}(Z_0))\| < 1$ and $Z_0$ has no pure imaginary eigenvalues~\cite[Lemma 5.12]{higham2008functions}.  In other words, $\|Z_k^2 - I\|$ provides an estimate for the accuracy of $Z_k$.  

For the iterates $Z_k$ with initial condition~(\ref{Z0}), we have, in the notation of\linebreak Lemma~\ref{lemma:YkOmegak},
\[
Z_k^2 - I = \begin{pmatrix} H_k^2 - I & H_k \Omega_k - \Omega_k H_k \\ 0 & H_k^2-I \end{pmatrix}.
\]
Likewise, for the iterates $Z_k$ with initial condition~(\ref{Z0b}), we have, in the notation of Lemma~\ref{lemma:YkSk},
\[
Z_k^2 - I = \begin{pmatrix} H_k^2 - I & H_k S_k + S_k H_k \\ 0 & H_k^2-I \end{pmatrix}.
\]
Thus, accuracy is assured when the quantities $\|H_k^2 - I\|$, $\|H_k \Omega_k - \Omega_k H_k\|$, and $\|H_k S_k + S_k H_k\|$ are small.  Of course, $H_k$, $\Omega_k$, and $S_k$ are never computed explicitly in the iteration~(\ref{Xupdate2}-\ref{Eupdate2}), so we must relate these quantities to $X_k$ and $E_k$ using Lemma~\ref{lemma:relation}.  
By~(\ref{relationXY}), we have
\[
H_k^2 - I = X_k^* X_k - I.
\]
The quantities $H_k \Omega_k - \Omega_k H_k$ and $H_k S_k + S_k H_k$ are more difficult to relate to $X_k$ and $E_k$ in a computable way (i.e., a way that does not involve knowing $U$ in advance).  However, second-order accurate approximations to $H_k \Omega_k - \Omega_k H_k$ and $H_k S_k + S_k H_k$ are available.  As shown in Appendix~\ref{sec:appendix}, we have
\begin{align}
H_k \Omega_k - \Omega_k H_k &= \frac{1}{2} \left( X_k^* X_k X_k^* E_k - X_k^* E_k X_k^* X_k \right) + F_k, \label{HWapprox} \\
H_k S_k - S_k H_k &= X_k^* E_k + E_k^* X_k - \frac{1}{2} \left( X_k^* X_k X_k^* E_k - X_k^* E_k X_k^* X_k \right) - F_k, \label{HSapprox}
\end{align}
where
\[
\|F_k\| = O\left(\|H_k^2-I\|^2+\|H_k^2-I\| \|H_k S_k + S_k H_k\| \right).
\]
Roughly speaking,~(\ref{HWapprox}) arises from the approximations $H_k \approx \frac{1}{2} (I + X_k^* X_k)$ and $\Omega_k \approx X_k^* E_k$.  It turns out that only the first of these approximations is second-order accurate (see Lemma~\ref{lemma:XTXapprox}), but delicate cancellations detailed in Appendix~\ref{sec:appendix} lead to the validity of~(\ref{HWapprox}).  One then deduces~(\ref{HSapprox}) by noting that $X_k^* E_k + E_k^* X_k = (H_k \Omega_k - \Omega_k H_k) + (H_k S_k - S_k H_k)$ (see Lemma~\ref{lemma:XE}).

In summary, the quantities
\begin{align}
\alpha_k &= X_k^* X_k - I, \label{alphak} \\
\beta_k &=\frac{1}{2}\left( X_k^* X_k X_k^* E_k - X_k^* E_k X_k^* X_k \right), \label{betak} \\
\gamma_k &= X_k^* E_k + E_k^* X_k - \beta_k \label{gammak}
\end{align}
are computable approximations to $H_k^2-I$, $H_k\Omega_k-\Omega_k H_k$, and $H_k S_k + S_k H_k$, respectively. These are small in norm if and only if $\|Z_k-\mathrm{sign}(Z_0)\|$ is small (for each of the initial conditions~(\ref{Z0}) and~(\ref{Z0b})), which is true if and only if $\|X_k - U\|$ and $\|E_k - L_{\mathcal{P}}(A,E)\|$ are small.  As a practical note, these arguments appear to break down if $A$ is very ill-conditioned, as illustrated in Section~\ref{sec:numerical}.

Based on these considerations, we propose that the iterations be terminated when
\begin{align}
\|\alpha_k\| \le \delta \|X_k\| \quad \text{ and } \quad \|\beta_k\| + \|\gamma_k\| \le \varepsilon \|E_k\|, \label{terminationcriteria}
\end{align}
where $\delta$ and $\varepsilon$ are relative error tolerances for $\|X_k - U\|$ and $\|E_k - L_{\mathcal{P}}(A,E)\|$, respectively.  

As an alternative approach to terminating the iterations, one could consider basing the decision to terminate on the smallness of the step lengths $\|X_{k+1}-X_k\|$ and $\|E_{k+1}-E_k\|$.  Details of this approach, for the case in which $E_k$ is absent, can be found in~\cite[Chapter 8.7]{higham2008functions}.

\subsection{Stability}

Stability of the iterative schemes detailed in Theorem~\ref{thm:dpolariter} is relatively easy to establish.  Indeed, the map 
\begin{equation} \label{dpolarmap}
\mathcal{F}\begin{pmatrix} A \\ E \end{pmatrix} = \begin{pmatrix} \mathcal{P}(A) \\ L_{\mathcal{P}}(A,E) \end{pmatrix}
\end{equation}
is idempotent, since $\mathcal{P}(\mathcal{P}(A))=\mathcal{P}(A)$ and $L_{\mathcal{P}}(\mathcal{P}(A),L_{\mathcal{P}}(A,E)) = L_{\mathcal{P}}(A,E)$ by the chain rule.  It follows that any superlinearly convergent iteration for computing $\begin{pmatrix} \mathcal{P}(A) \\ L_{\mathcal{P}}(A,E) \end{pmatrix}$ is automatically stable~\cite[Therorem 4.19]{higham2008functions}.  More precisely, if 
\begin{equation} \label{itermap}
\begin{pmatrix} X_{k+1} \\ E_{k+1} \end{pmatrix} = \begin{pmatrix} g(X_k) \\ L_g(X_k,E_k) \end{pmatrix}
\end{equation}
converges superlinearly to $\begin{pmatrix} \mathcal{P}(X_0) \\ L_{\mathcal{P}}(X_0,E_0) \end{pmatrix}$ for all $X_0$ and $E_0$ sufficiently close to $A$ and $E$, respectively, then the iteration is stable in the sense of~\cite[Definition 4.17]{higham2008functions}.
Moreover, the Fr\'{e}chet derivative of the map~(\ref{dpolarmap}) coincides with the Fr\'{e}chet derivative of the map~(\ref{itermap}) at the fixed point $\begin{pmatrix} \mathcal{P}(A) \\ L_{\mathcal{P}}(A,E) \end{pmatrix}$~\cite[Therorem 4.19]{higham2008functions}.

As an example, the Newton iteration~(\ref{XupdateNewton}-\ref{EupdateNewton}) is superlinearly convergent by virtue of the superlinear (indeed, quadratic) convergence of the corresponding matrix sign function iteration~(\ref{signNewton}).
The Newton-Schulz iteration~(\ref{XupdateNewtonSchulz}-\ref{EupdateNewtonSchulz}) is likewise superlinearly (indeed, quadratically) convergent, provided that the singular values of $A$ lie in the interval $(0,\sqrt{2})$.  Thus, both iterations are stable.  Using, for instance,~(\ref{XupdateNewtonSchulz}-\ref{EupdateNewtonSchulz}), we find that the Fr\'{e}chet derivative of the map~(\ref{itermap}) (and hence of the map~(\ref{dpolarmap})) at $\begin{pmatrix} U \\ K \end{pmatrix} = \begin{pmatrix} \mathcal{P}(A) \\ L_{\mathcal{P}}(A,E) \end{pmatrix}$ is given by
\[
L_{\mathcal{F}}\Big( \begin{pmatrix} U \\ K \end{pmatrix}, \begin{pmatrix} F \\ G \end{pmatrix} \Big) = \begin{pmatrix} F - \frac{1}{2} U(U^*F+F^*U) \\ G - \frac{1}{2} \left[ U(U^*G+G^*U) + K(U^*F+F^*U) + U (K^* F + F^*K)  \right] \end{pmatrix}.
\]
Note that when $U$ is square, the identities $UU^*=I$ and $U^*K=-K^*U$ (by~(\ref{symLfo})) imply that this formula reduces to
\[
L_{\mathcal{F}}\left( \begin{pmatrix} U \\ K \end{pmatrix}, \begin{pmatrix} F \\ G \end{pmatrix} \right) = \begin{pmatrix} \frac{1}{2} (F - U F^*U) \\ \frac{1}{2} (G - U G^*U - U F^*K - K F^* U) \end{pmatrix},
\]
in agreement with~\cite[Theorem 8.19]{higham2008functions}.

\subsection{Condition Number Estimation} \label{sec:condition}

A seemingly natural application of \linebreak Theorem~\ref{thm:dpolariter} is to leverage the iterative scheme~(\ref{Xupdate2}-\ref{Eupdate2}) to estimate the condition number 
\[
\kappa(\mathcal{P},A) = \|L_{\mathcal{P}}(A,\cdot)\| = \sup_{E \in \mathbb{C}^{m \times n}, \atop E \neq 0} \frac{\|L_{\mathcal{P}}(A,E)\|}{\|E\|}
\] 
of the map $\mathcal{P}$ at $A$.  As tempting as it may seem, a much simpler (and undoubtedly more efficient) algorithm is available for estimating $\kappa(\mathcal{P},A)$.  As explained in~\cite[Theorem 8.9]{higham2008functions}, the value of $\kappa(\mathcal{P},A)$ at $A \in \mathbb{C}^{m \times n}$ ($m \ge n$) is $\sigma_n^{-1}$, where $\sigma_n$ denotes the smallest singular value of $A$.  This quantity can be estimated efficiently by applying the power method~\cite[Algorithm 3.19]{higham2008functions} to $(A^*A)^{-1}$.  In most iterative algorithms for computing the polar decomposition, this matrix (or $A^{-1}$) is computed in the first iteration, so the additional cost of computing $\kappa(\mathcal{P},A)$ is negligible.  

Before finishing our discussion of condition number estimation, it is worth pointing out 
a subtlety that arises when considering the polar decomposition of a real square matrix.  If $A$ is real and square ($m=n$), then it can be shown that the condition number of $A$ with respect to \emph{real} perturbations is $2(\sigma_n + \sigma_{n-1})^{-1}$~\cite[Theorem 8.9]{higham2008functions}.  This fact will play a role in our interpretation of certain numerical experiments in Section~\ref{sec:numerical}.

\section{Comparison with Other Methods} \label{sec:comparison}

There are several other methods that can be used to compute the Fr\'{e}chet derivative of the polar decomposition.  Below, we describe a few and compare them with iterative schemes of the form~(\ref{Xupdate2}-\ref{Eupdate2}).

One alternative is to recognize that $L_{\mathcal{P}}(A,E)$ is the solution to a Lyapunov equation.  Indeed, upon noting that $\mathcal{P}(A)^* A=U^*A = H$ is Hermitian, one can differentiate the relation
\[
\mathrm{skew}(\mathcal{P}(A)^* A) = 0
\]
with the aid of the product rule to obtain
\[
\mathrm{skew}(L_{\mathcal{P}}(A,E)^* A + \mathcal{P}(A)^* E) = 0
\]
for any $E \in \mathbb{C}^{m \times n}$.  Substituting $A=UH$ and $\mathcal{P}(A)=U$, and denoting $Y := U^* L_{\mathcal{P}}(A,E) = -L_{\mathcal{P}}(A,E)^* U$, we obtain
\begin{equation} \label{syl}
H Y + Y H = U^* E - E^* U.
\end{equation}
Given $H$, $U$, and $E$, this is a Lyapunov equation in the unknown $Y$, which, by the positive-definiteness of $H$, has a unique solution.  It can be solved using standard algorithms for the solution of Lyapunov and Sylvester equations~\cite{bartels1972solution,golub1979hessenberg}.  It also has theoretical utility, offering an alternative proof of part of Theorem~\ref{thm:signblockmat}, owing to a well-known connection between the solution of Lyapunov and Sylvester equations and the matrix sign function~\cite[Chapter 2.4]{roberts1980linear,higham2008functions}.  Indeed,~(\ref{syl}) is equivalent to the equation
\[
\begin{pmatrix} H & E^* U - U^* E \\ 0 & -H \end{pmatrix} = \begin{pmatrix} I & Y \\ 0 & I \end{pmatrix} \begin{pmatrix} H & 0 \\ 0 & -H \end{pmatrix} \begin{pmatrix} I & Y \\ 0 & I \end{pmatrix}^{-1}.
\]
Taking the sign of both sides, noting that $\mathrm{sign}(H)=I$, and using the fact that the matrix sign function commutes with similarity transformations, we conclude that
\begin{align*}
\mathrm{sign} \begin{pmatrix} H & E^* U - U^* E \\ 0 & -H \end{pmatrix} &= \begin{pmatrix} I & Y \\ 0 & I \end{pmatrix} \left[\mathrm{sign} \begin{pmatrix} H & 0 \\ 0 & -H \end{pmatrix}\right] \begin{pmatrix} I & Y \\ 0 & I \end{pmatrix}^{-1} \\
&= \begin{pmatrix} I & -2Y \\ 0 & -I \end{pmatrix}.
\end{align*}
This is precisely the identity~(\ref{signHW0}), up to a rescaling of $E$.  Its connection with the Lyapunov equation~(\ref{syl}) reveals that the coupled iteration~(\ref{Xupdate2}-\ref{Eupdate2}) is effectively solving~(\ref{syl}) and computing the polar decomposition simultaneously.  In comparison to a naive approach in which~(\ref{syl}) is solved after first computing the polar decomposition, the coupled iteration~(\ref{Xupdate2}-\ref{Eupdate2}) is attractive, as it computes $L_{\mathcal{P}}(A,E)$ at the expense of a few extra matrix-matrix multiplications and additions on top of the computation of $\mathcal{P}(A)$.

When $A$ and $E$ are real, another method for computing Fr\'{e}chet derivative of a matrix function $f$ is to use the complex step approximation~\cite{al2010complex}
\[
L_f(A,E) \approx \mathrm{Im} \left( \frac{f(A+ihE)-f(A)}{h} \right),
\]
where $h$ is a small positive scalar and $\mathrm{Im}(B)$ denotes the imaginary part of a matrix $B$.  By using a pure imaginary step $ih$, this approximation does not suffer from cancellation errors that plague standard finite differencing, allowing $h$ to be taken arbitrarily small~\cite{al2010complex}.  This approximation can be applied to the polar decomposition, but care must be exercised in order to do so correctly.  In particular, a meaningful approximation is obtained only if the conjugate transposes $X_k^*$ appearing in the algorithm are interpreted as transposes $X_k^T$ when evaluating the ``polar decomposition'' of $A+ihE$.
We have put ``polar decomposition'' in quotes since the result of such a computation is the matrix $(A+ihE) \left[(A+ihE)^T (A+ihE)\right]^{-1/2}$, not $\mathcal{P}(A+ihE) = (A+ihE) \left[(A+ihE)^* (A+ihE)\right]^{-1/2}$.
The cost of this approximation is close to the cost of computing two polar decompositions.  

Another approach is to appeal to the relation $\mathcal{P}(A)=A(A^*A)^{-1/2}$.  By~(\ref{Lg}), the Fr\'{e}chet derivative of $\mathcal{P}$ at $A$ in the direction $E$ is given by
\begin{align*}
L_{\mathcal{P}}(A,E) 
&= E (A^*A)^{-1/2} + AL_{x^{-1/2}} (A^*A, E^*A + A^*E) \\
&= E H^{-1} + AL_{x^{-1/2}} (A^*A, E^*A + A^*E).
\end{align*}
Evaluating the second term, the Fr\'{e}chet derivative of the inverse square root, can be reduced to the task of solving a Lyapunov equation, so this approach is essentially of the same complexity as the one based on~(\ref{syl}).

Any of the aforementioned methods, including our own, 
can be applied in two different ways when $A$ is rectangular ($m \times n$ with $m>n$).  One way is to apply the methods verbatim, working at all times with rectangular matrices.  The alternative is to first compute a reduced $QR$ decomposition $A=QR$, where $Q \in \mathbb{C}^{m \times n}$ has orthonormal columns and $R \in \mathbb{C}^{n \times n}$ is upper triangular.  Then, one can compute $\mathcal{P}(R)$ and $L_{\mathcal{P}}(R,Q^*E)$ (which are square matrices) and invoke the identities
\[
U = \mathcal{P}(A) = Q\mathcal{P}(R), \quad H = \mathcal{P}(R)^*R
\]
and
\begin{align*}
L_{\mathcal{P}}(A,E) 
&= L_{\mathcal{P}}(A,QQ^*E) + L_{\mathcal{P}}(A,(I-QQ^*)E) \\
&= Q L_{\mathcal{P}}(R,Q^*E) + (I-QQ^*)EH^{-1}
\end{align*}
to recover $\mathcal{P}(A)$ and $L_{\mathcal{P}}(A,E)$.  The validity of the latter identity is a consequence of~(\ref{dpolarsymmetry}),~(\ref{dpolarperp}), and the fact that $QQ^*=UU^*$.  In summary, computations for rectangular $A$ can be reduced to the square case by performing a reduced $QR$ decomposition of $A$ at the outset.

Finally, when $A$ is square, one more method for computing $L_{\mathcal{P}}(A,E)$ is available, as noted in, for instance,~\cite{kenney1991polar}.  The idea is to make use of the singular value decomposition $A=P\Sigma Q^*$, where $P,Q \in \mathbb{C}^{n \times n}$ are unitary and $\Sigma \in \mathbb{C}^{n \times n}$ is diagonal.  The singular value decomposition is related to the polar decomposition $A=UH$ via the relations $U = PQ^*$ and $H=Q\Sigma Q^*$.  Moreover, the Lyapunov equation~(\ref{syl}) is equivalent to
\[
\Sigma G + G \Sigma =  F - F^*,
\]
where $F = P^* E Q$ and $G = P^* L_{\mathcal{P}}(A,E) Q$~\cite[Equation 2.18]{kenney1991polar}.  Given $\Sigma$ and $F$, this equation admits an explicit solution for the components of $G$. Namely,
\[
G_{ij} = \frac{1}{\sigma_i+\sigma_j} (F_{ij}-\overline{F_{ji}}),
\]
where $\sigma_i$ denotes the $i^{th}$ diagonal entry of $\Sigma$, and $\overline{F_{ji}}$ denotes the complex conjugate of $F_{ji}$.  One then obtains $L_{\mathcal{P}}(A,E)$ from $L_{\mathcal{P}}(A,E) = PGQ^*$.  This method is attractive if the singular value decomposition of $A$ has already been computed, but otherwise it is an expensive approach in general. 

\subsection{Floating Point Operations}  
Relative to the methods listed above, the iterative schemes derived in this paper are distinguished by their efficiency, at least when $n$ is large and the columns of $A$ are close to being orthonormal.  To see this, consider the number of floating point operations needed to compute $\mathcal{P}(A)$ and $L_{\mathcal{P}}(A,E)$.  For simplicity, assume that $A$ and $E$ are real and of size $n \times n$.  Then, to leading order in $n$, and excluding the costs associated with termination criteria in the iterative schemes, the methods have the following computational costs:

\begin{itemize}
\item The iteration~(\ref{XupdateNewton}-\ref{EupdateNewton}) requires $n_{iter}$ matrix inversions (each requiring $2n^3$ flops~\cite[Appendix C]{higham2008functions}) and $2n_{iter}$ matrix multiplications (each requiring $2n^3$ flops), where $n_{iter}$ denotes the number of iterations used.  Its computational cost is thus $n_{iter}(2n^3)+2n_{iter}(2n^3) = 6n_{iter} n^3$ flops.
\item Solving the Lyapunov equation~(\ref{syl}) with a direct method involves diagonalizing $H$ ($9n^3$ flops~\cite[Appendix C]{higham2008functions}) and performing 4 matrix multiplications, for a total of $9n^3+4(2n^3)=17n^3$ flops.  The additional cost of computing $U$, $H=U^*A$, $L_{\mathcal{P}}(A,E) = UY$, and $U^*E$ (assuming that~(\ref{XupdateNewton}) is used to compute $U$) is dominated by the cost of performing $n_{iter}$ matrix inversions and 3 matrix multiplications, bringing the total to $17n^3 + n_{iter} (2n^3) + 3(2n^3) = (23+2n_{iter}) n^3$ flops.
\item The complex step approximation (assuming that~(\ref{XupdateNewton}) is used to compute the polar decomposition of $A$ and $A+ihE$)  requires $2n_{iter}$ matrix inversions, of which $n_{iter}$ involve complex arithmetic.  Since each inversion of a complex matrix requires $n^3$ additions of complex scalars (2 real flops) and $n^3$ multiplications of complex scalars (6 real flops), the computational cost of the complex step approximation is $n_{iter}(2n^3) + n_{iter}(8n^3) = 10n_{iter} n^3$ flops.
\item The method based on the singular value decomposition requires 5 matrix multiplications plus the computation of the SVD.  Assuming, for instance, that the Golub-Reinsch algoirthm ($22n^3$ flops~\cite{golub2012matrix}) is used to compute the SVD, this method's total cost is $5(2n^3) + 22n^3 = 32n^3$ flops.
\end{itemize}
We conclude from this analysis that, for sufficienty large $n$, the iteration~(\ref{XupdateNewton}-\ref{EupdateNewton}) requires fewer floating point operations than its competitors whenever $n_{iter} \le 5$.  Note that this is no longer the case if the costs of computing the residual estimates~(\ref{alphak}-\ref{gammak}) are taken into account.  However, if efficiency is the primary objective, then cheaper termination criteria (based, for instance, on $\|X_k-X_k^{-*}\|$, $\|X_{k+1}-X_k\|$, and/or $\|E_{k+1}-E_k\|$) may be appropriate.


\section{Numerical Experiments} \label{sec:numerical}

\begin{table}[t]
\centering
\resizebox{\columnwidth}{!}{%
\begin{filecontents*}{Data/moler.dat}
   1.0000000e+00   1.3461809e-05   2.6012253e-03   1.0769703e-04   1.9711451e-04   5.2425173e-02   1.9710172e-04   5.2425173e-02   9.9999755e-01
   2.0000000e+00   2.9618501e-10   5.8891933e-08   2.3694791e-09   4.3440029e-09   1.1868830e-06   4.3440015e-09   1.1868830e-06   1.0000000e+00
   3.0000000e+00   8.0422935e-16   1.1621045e-15   1.4244467e-15   2.5106055e-15   4.1304371e-15   1.1015285e-14   1.1767578e-14   1.0000000e+00
\end{filecontents*}
\pgfplotstabletypeset[
every head row/.style={after row=\midrule},
columns={0,1,2,3,4,5,6,7,8},
columns/0/.style={fixed,column type/.add={}{|},column name={$k$}},
columns/1/.style={sci,sci e,sci zerofill,precision=1,column type/.add={}{|},column name={$\frac{\|X_k-U\|}{\|U\|}$}},
columns/2/.style={sci,sci e,sci zerofill,precision=1,column type/.add={}{|},column name={$\frac{\|E_k-K\|}{\|K\|}$}}, 
columns/3/.style={sci,sci e,sci zerofill,precision=1,column type/.add={}{|},column name={$\|\alpha_k\|$}}, 
columns/4/.style={sci,sci e,sci zerofill,precision=1,column type/.add={}{|},column name={$\|\beta_k\|$}}, 
columns/5/.style={sci,sci e,sci zerofill,precision=1,column type/.add={}{|},column name={$\|\gamma_k\|$}}, 
columns/6/.style={sci,sci e,sci zerofill,precision=1,column type/.add={}{|},column name={$\|\widetilde{\beta}_k\|$}}, 
columns/7/.style={sci,sci e,sci zerofill,precision=1,column type/.add={}{|},column name={$\|\widetilde{\gamma}_k\|$}}, 
columns/8/.style={sci,sci e,sci zerofill,precision=1,column type/.add={}{},column name={$\mu_k$}}, 
]
{Data/moler.dat}
}
\caption{Nearly orthogonal matrix, $m=n=16$, $\sigma_n(A) = 9.9e{-1}$, $\sigma_{n-1}(A)=1.0e{+0}$, $\kappa(A) = 1.0e{+0}$.}
\label{tab:moler}
\end{table}

\begin{table}[t]
\centering
\resizebox{\columnwidth}{!}{%
\begin{filecontents*}{Data/binomial.dat}
   1.0000000e+00   1.6943744e+01   3.3468749e+01   2.3701791e+03   1.1446025e+05   1.1426918e+05   1.9040995e+02   6.7383815e+02   1.7629399e-01
   2.0000000e+00   1.4088534e+00   2.3159944e+00   2.2084852e+01   7.6623110e+00   8.0873702e+00   1.4137930e+00   5.2328608e+00   5.8517144e-01
   3.0000000e+00   1.3071474e-01   3.4254967e-01   1.1291301e+00   3.2860534e-02   4.0624647e-01   2.5930665e-02   4.0631478e-01   9.2009503e-01
   4.0000000e+00   2.6207534e-03   2.1868598e-02   2.1002407e-02   8.5893632e-04   2.5274294e-02   8.5467191e-04   2.5274333e-02   9.9827262e-01
   5.0000000e+00   1.3568786e-06   4.4945083e-05   1.0855039e-05   6.5870003e-07   5.2336858e-05   6.5869850e-07   5.2336858e-05   9.9999896e-01
   6.0000000e+00   3.9045757e-13   5.1410521e-11   3.1165249e-12   1.0174714e-13   5.9930966e-11   1.0947147e-13   5.9913758e-11   1.0000000e+00
   7.0000000e+00   2.8219867e-14   6.4660349e-14   1.3643301e-15   1.0646980e-16   1.9615976e-16   4.0252059e-14   5.7395101e-14   1.0000000e+00
\end{filecontents*}
\pgfplotstabletypeset[
every head row/.style={after row=\midrule},
columns={0,1,2,3,4,5,6,7,8},
columns/0/.style={fixed,column type/.add={}{|},column name={$k$}},
columns/1/.style={sci,sci e,sci zerofill,precision=1,column type/.add={}{|},column name={$\frac{\|X_k-U\|}{\|U\|}$}},
columns/2/.style={sci,sci e,sci zerofill,precision=1,column type/.add={}{|},column name={$\frac{\|E_k-K\|}{\|K\|}$}}, 
columns/3/.style={sci,sci e,sci zerofill,precision=1,column type/.add={}{|},column name={$\|\alpha_k\|$}}, 
columns/4/.style={sci,sci e,sci zerofill,precision=1,column type/.add={}{|},column name={$\|\beta_k\|$}}, 
columns/5/.style={sci,sci e,sci zerofill,precision=1,column type/.add={}{|},column name={$\|\gamma_k\|$}}, 
columns/6/.style={sci,sci e,sci zerofill,precision=1,column type/.add={}{|},column name={$\|\widetilde{\beta}_k\|$}}, 
columns/7/.style={sci,sci e,sci zerofill,precision=1,column type/.add={}{|},column name={$\|\widetilde{\gamma}_k\|$}}, 
columns/8/.style={sci,sci e,sci zerofill,precision=1,column type/.add={}{},column name={$\mu_k$}}, 
]
{Data/binomial.dat}
}
\caption{Binomial matrix, $m=n=16$, $\sigma_n(A)=2.6e{+0}$, $\sigma_{n-1}(A)=2.6e{+0}$, $\kappa(A) =4.7e{+3}$.}
\label{tab:binomial}
\end{table}

\begin{table}[t]
\centering
\resizebox{\columnwidth}{!}{%
\begin{filecontents*}{Data/frank.dat}
   1.0000000e+00   2.8581145e+06   6.0572994e+18   8.3628547e+13   3.1932841e+27   3.2252508e+27   1.3666497e+14   4.5309746e+26   1.0565871e-06
   2.0000000e+00   1.8535026e+00   3.1571311e+12   4.5184652e+01   1.0911395e+03   1.2643409e+14   4.0442064e+01   1.2643409e+14   4.1648239e-01
   3.0000000e+00   2.6575765e-01   4.5953216e+11   2.5132034e+00   2.8080147e+00   4.9858872e+12   1.5875822e+00   4.9858872e+12   8.3449169e-01
   4.0000000e+00   9.5285293e-03   4.1291684e+09   7.6834205e-02   9.2756902e-02   3.7788749e+10   9.0682603e-02   3.7788749e+10   9.9403711e-01
   5.0000000e+00   3.8911665e-05   2.4459921e+07   3.1131225e-04   4.7492300e-04   2.2384014e+08   4.7484681e-04   2.2384014e+08   9.9997112e-01
   6.0000000e+00   1.5773839e-09   2.7732889e+02   1.2619070e-08   2.3520047e-08   2.5378760e+03   2.3520072e-08   2.5378760e+03   1.0000000e+00
   7.0000000e+00   3.3910956e-15   1.8257421e-04   1.0208925e-15   9.7674697e-16   2.6229015e-06   1.0138651e-14   2.6229015e-06   1.0000000e+00
   8.0000000e+00   3.3868680e-15   1.8257399e-04   9.2877636e-16   8.4090784e-16   1.8294765e-15   1.0068515e-14   1.0638301e-14   1.0000000e+00
\end{filecontents*}
\pgfplotstabletypeset[
every head row/.style={after row=\midrule},
columns={0,1,2,3,4,5,6,7,8},
columns/0/.style={fixed,column type/.add={}{|},column name={$k$}},
columns/1/.style={sci,sci e,sci zerofill,precision=1,column type/.add={}{|},column name={$\frac{\|X_k-U\|}{\|U\|}$}},
columns/2/.style={sci,sci e,sci zerofill,precision=1,column type/.add={}{|},column name={$\frac{\|E_k-K\|}{\|K\|}$}}, 
columns/3/.style={sci,sci e,sci zerofill,precision=1,column type/.add={}{|},column name={$\|\alpha_k\|$}}, 
columns/4/.style={sci,sci e,sci zerofill,precision=1,column type/.add={}{|},column name={$\|\beta_k\|$}}, 
columns/5/.style={sci,sci e,sci zerofill,precision=1,column type/.add={}{|},column name={$\|\gamma_k\|$}}, 
columns/6/.style={sci,sci e,sci zerofill,precision=1,column type/.add={}{|},column name={$\|\widetilde{\beta}_k\|$}}, 
columns/7/.style={sci,sci e,sci zerofill,precision=1,column type/.add={}{|},column name={$\|\widetilde{\gamma}_k\|$}}, 
columns/8/.style={sci,sci e,sci zerofill,precision=1,column type/.add={}{},column name={$\mu_k$}}, 
]
{Data/frank.dat}
}
\caption{Frank matrix, $m=n=16$, $\sigma_n(A)=3.5e{-13}$, $\sigma_{n-1}(A)=8.7e{-1}$, $\kappa(A)=2.3e{+14}$.}
\label{tab:frank}
\end{table}

\begin{table}[t]
\centering
\resizebox{\columnwidth}{!}{%
\begin{filecontents*}{Data/frankmodified.dat}
   1.0000000e+00   3.4092242e+06   8.0116351e+06   9.9871807e+13   8.2628713e+36   8.2628713e+36   1.9487390e+23   2.8953406e+26   6.4860959e-07
   2.0000000e+00   1.2380367e+00   1.8903541e+00   2.1820257e+01   1.0152306e+11   3.0338274e+13   2.0431072e+10   3.0357615e+13   5.1201893e-01
   3.0000000e+00   1.3868816e-01   8.4977582e-02   1.2371579e+00   5.4645660e+08   7.9935797e+11   5.3747145e+08   8.0654016e+11   9.4109974e-01
   4.0000000e+00   7.2419231e-03   3.6908647e-03   5.8418103e-02   2.7893136e+07   4.1320467e+10   2.7816435e+07   3.4886875e+10   9.9815330e-01
   5.0000000e+00   3.6599545e-04   1.7948411e-03   4.8581583e-04   1.3049152e+04   1.9500436e+07   1.3049141e+04   9.6393186e+09   9.9999495e-01
   6.0000000e+00   3.6092346e-04   1.7940717e-03   5.4350178e-08   6.5442106e-02   9.6265362e+01   6.6928708e-02   9.6250487e+09   1.0000000e+00
   7.0000000e+00   3.6092346e-04   1.7940717e-03   1.1564798e-15   9.1240387e-04   1.8863258e-03   6.4195930e-03   9.6250487e+09   1.0000000e+00
\end{filecontents*}
\pgfplotstabletypeset[
every head row/.style={after row=\midrule},
columns={0,1,2,3,4,5,6,7,8},
columns/0/.style={fixed,column type/.add={}{|},column name={$k$}},
columns/1/.style={sci,sci e,sci zerofill,precision=1,column type/.add={}{|},column name={$\frac{\|X_k-U\|}{\|U\|}$}},
columns/2/.style={sci,sci e,sci zerofill,precision=1,column type/.add={}{|},column name={$\frac{\|E_k-K\|}{\|K\|}$}}, 
columns/3/.style={sci,sci e,sci zerofill,precision=1,column type/.add={}{|},column name={$\|\alpha_k\|$}}, 
columns/4/.style={sci,sci e,sci zerofill,precision=1,column type/.add={}{|},column name={$\|\beta_k\|$}}, 
columns/5/.style={sci,sci e,sci zerofill,precision=1,column type/.add={}{|},column name={$\|\gamma_k\|$}}, 
columns/6/.style={sci,sci e,sci zerofill,precision=1,column type/.add={}{|},column name={$\|\widetilde{\beta}_k\|$}}, 
columns/7/.style={sci,sci e,sci zerofill,precision=1,column type/.add={}{|},column name={$\|\widetilde{\gamma}_k\|$}}, 
columns/8/.style={sci,sci e,sci zerofill,precision=1,column type/.add={}{},column name={$\mu_k$}}, 
]
{Data/frankmodified.dat}
}
\caption{Modified Frank matrix, $m=n=16$,  $\sigma_n(A)=3.5e{-13}$, $\sigma_{n-1}(A)=3.5e{-13}$, $\kappa(A)=2.3e{+14}$.}
\label{tab:frankmodified}
\end{table}

\begin{table}[t]
\centering
\resizebox{\columnwidth}{!}{%
\begin{filecontents*}{Data/rect.dat}
   1.0000000e+00   1.8136597e+00   3.2792700e+00   1.9750019e+01   6.9009742e+01   1.8322589e+02   1.1919938e+01   1.7263816e+02   5.0105221e-01
   2.0000000e+00   2.2231467e-01   3.9598830e-01   1.1060292e+00   1.2682934e-01   9.9940797e+00   8.6183335e-02   9.9933717e+00   8.2067705e-01
   3.0000000e+00   1.7673390e-04   1.5930342e-03   7.9048776e-04   4.2976369e-04   3.9280767e-02   4.2961179e-04   3.9280763e-02   9.9984771e-01
   4.0000000e+00   8.6920647e-09   2.1599258e-07   3.8872097e-08   1.9322795e-08   5.3227016e-06   1.9322796e-08   5.3227016e-06   9.9999999e-01
   5.0000000e+00   3.7761563e-15   1.3269599e-14   4.7868311e-16   2.6798261e-16   9.2300879e-15   7.4624064e-15   3.0024580e-14   1.0000000e+00
\end{filecontents*}
\pgfplotstabletypeset[
every head row/.style={after row=\midrule},
columns={0,1,2,3,4,5,6,7,8},
columns/0/.style={fixed,column type/.add={}{|},column name={$k$}},
columns/1/.style={sci,sci e,sci zerofill,precision=1,column type/.add={}{|},column name={$\frac{\|X_k-U\|}{\|U\|}$}},
columns/2/.style={sci,sci e,sci zerofill,precision=1,column type/.add={}{|},column name={$\frac{\|E_k-K\|}{\|K\|}$}}, 
columns/3/.style={sci,sci e,sci zerofill,precision=1,column type/.add={}{|},column name={$\|\alpha_k\|$}}, 
columns/4/.style={sci,sci e,sci zerofill,precision=1,column type/.add={}{|},column name={$\|\beta_k\|$}}, 
columns/5/.style={sci,sci e,sci zerofill,precision=1,column type/.add={}{|},column name={$\|\gamma_k\|$}}, 
columns/6/.style={sci,sci e,sci zerofill,precision=1,column type/.add={}{|},column name={$\|\widetilde{\beta}_k\|$}}, 
columns/7/.style={sci,sci e,sci zerofill,precision=1,column type/.add={}{|},column name={$\|\widetilde{\gamma}_k\|$}}, 
columns/8/.style={sci,sci e,sci zerofill,precision=1,column type/.add={}{},column name={$\mu_k$}}, 
]
{Data/rect.dat}
}
\caption{Rectangular matrix, $m=16$, $n=5$,  $\sigma_n(A)=2.5e{-1}$, $\sigma_{n-1}(A)=2.3e{+0}$, $\kappa(A)=5.8e{+1}$.}
\label{tab:rect}
\end{table}

To illustrate the performance of the iterative schemes derived in this paper, we have computed the Fr\'{e}chet derivative of the polar decomposition for the following matrices obtained from MATLAB's matrix gallery.  Note that the first three matrices are identical to those considered in~\cite[Chapter 8.9]{higham2008functions}.
\begin{enumerate}
\item A nearly orthogonal matrix, \verb$orth(gallery('moler',16))+ones(16)*1e-3$.
\item A binomial matrix, \verb+gallery('binomial',16)+.
\item The Frank matrix, \verb$gallery('frank',16)$.
\item A modification of the Frank matrix obtained by setting its second smallest singular value equal to its smallest singular value.  That is, $A = P \widetilde{\Sigma} Q^*$ where $P \Sigma Q^*$ is the singular value decomposition of the Frank matrix, $\widetilde{\Sigma}_{ii} = \Sigma_{ii}$ for $i \neq 15$, and $\widetilde{\Sigma}_{15,15} = \Sigma_{16,16}$.
\item A rectangular matrix given by the first 5 columns of the binomial matrix.
\end{enumerate}
We computed $\mathcal{P}(A)$ and $L_{\mathcal{P}}(A,E)$ for each $A$ listed above, with $E$ a matrix (of the same dimensions as $A$) consisting of random entries sampled from a normal distribution with mean 0 and variance 1.  We used the Newton iteration~(\ref{XupdateNewtonscaled}-\ref{EupdateNewtonscaled}) with scaling parameter~(\ref{scale1inf}) for the square matrices and its generalization~(\ref{XupdateNewtonrect}-\ref{EupdateNewtonrect}) with scaling parameter~(\ref{scale1infrect}) for the rectangular matrix.  To terminate the iterations, we used~(\ref{terminationcriteria}) with $\delta=\varepsilon=10^{-14}$ and $\|\cdot\|$ equal to the Frobenius norm.  To compute the ``exact'' values of $\mathcal{P}(A)$ and $L_{\mathcal{P}}(A,E)$, we used the singular value decomposition, as explained in the last paragraph of Section~\ref{sec:comparison}.

Note that for simplicity, we used scaling throughout the entire iteration, even though the scaling parameter $\mu_k$ approaches 1 near convergence.  A more efficient approach is to switch to an unscaled iteration after a certain point. A heuristic for deciding when to do so is detailed in~\cite[Chapter 8.9]{higham2008functions}.

Tables~\ref{tab:moler}-\ref{tab:rect} show the values of several quantities monitored during the iterations.  The first two columns show the relative errors $\frac{\|X_k - U\|}{\|U\|}$ and $\frac{\|E_k-K\|}{\|K\|}$, where $U=\mathcal{P}(A)$ and $K=L_{\mathcal{P}}(A,E)$.  The next three columns show the norms of~(\ref{alphak}-\ref{gammak}), which are the quantities we used to determine when to terminate the iterations.  Recall that~(\ref{betak}) and~(\ref{gammak}) are computable approximations to $H_k\Omega_k - \Omega_k H_k$ and $H_k S_k + S_k H_k$, respectively.  We have denoted $\widetilde{\beta}_k = H_k\Omega_k - \Omega_k H_k$ and $\widetilde{\gamma}_k = H_k S_k + S_k H_k$ in the tables and recorded their norms in the seventh and eighth columns. 
Finally, the last column of the tables shows the value of the scaling parameter $\mu_k$.  All norms appearing in the table headers are the Frobenius norm.  In the caption of each table, we have made note of the dimensions of the matrix $A$, the smallest and second smallest singular values $\sigma_n(A)$ and $\sigma_{n-1}(A)$ of $A$, respectively, and the condition number $\kappa(A)$ of $A$.

Tables~\ref{tab:moler}, \ref{tab:binomial}, and~\ref{tab:rect} illustrate the effectiveness of the iteration on relatively well-conditioned matrices.  In all three cases, small relative errors in both $X_k$ and $E_k$ are achieved simultaneously, and convergence is detected appropriately by the termination criteria~(\ref{terminationcriteria}).  Comparison of the columns labeled $\|\beta_k\|$ and $\|\gamma_k\|$ with the columns labeled $\|\widetilde{\beta}_k\|$ and $\|\widetilde{\gamma}_k\|$, respectively, lends credence to the asymptotic accuracy of the approximations $\beta_k \approx \widetilde{\beta}_k$ and $\gamma_k \approx \widetilde{\gamma}_k$, at least until roundoff errors begin to intervene.

Tables~\ref{tab:frank} and~\ref{tab:frankmodified} illustrate what can go wrong when $A$ is ill-conditioned.  In the case of Table~\ref{tab:frankmodified}, the matrix $A$ (the modified Frank matrix) has condition number $\kappa(A)=2.3e{+14}$, and its two smallest singular values are both close to zero: $\sigma_n(A)=\sigma_{n-1}(A)=3.5e{-13}$.  As a consequence, the condition number of $\mathcal{P}$ with respect to real perturbations (as explained in Section~\ref{sec:condition}) is  $2(\sigma_n + \sigma_{n-1})^{-1} = 2.9e{+12}$, and we cannot expect much more than 3 or 4 digits of relative accuracy in double precision arithmetic when approximating $\mathcal{P}(A)$, much less $L_{\mathcal{P}}(A,E)$.  This expectation is born out in Table~\ref{tab:frankmodified}.  A more subtle phenomenon occurs in Table~\ref{tab:frank}.  There, the matrix $A$ (the Frank matrix) has condition number $\kappa(A)=2.3e{+14}$ as well, but only one of its singular values is close to zero.  Namely, $\sigma_n(A)=3.5e{-13}$, but $\sigma_{n-1}(A)=8.7e{-1}$.  As a consequence, $\mathcal{P}$ is very well-conditioned with respect to real perturbations, having condition number $2(\sigma_n + \sigma_{n-1})^{-1} = 1.2e{+0}$.  Curiously, the result is that $\mathcal{P}(A)$ is approximated very accurately, but $L_{\mathcal{P}}(A,E)$ is not.  The fact that the performance of the Newton iteration~(\ref{XupdateNewtonscaled}) is largely unaffected by poorly conditioned $A$ (unless $A$ has two singular values close to zero) has been noted in~\cite[Chapter 8.9]{higham2008functions}.  The observation that, in contrast, it takes only one near-zero singular value to corrupt the computation of $L_{\mathcal{P}}(A,E)$ via the iteration~(\ref{XupdateNewtonscaled}-\ref{EupdateNewtonscaled}) deserves further study.

\section{Conclusion}

This paper has derived iterative schemes for computing the Fr\'{e}chet derivative of the polar decomposition.  The structure of these iterative schemes lends credence to the mantra that differentiating an iteration for computing $f(A)$ leads to an iteration for computing $L_f(A,E)$.  It would be interesting to determine what conditions on a matrix function $f$ ensure that this mantra bears out in practice.
Certainly being a primary matrix function suffices, but the results of the present paper suggest that such a construction might work in a more general setting.  

On a more specific level, several aspects of this paper warrant further consideration.  While the termination criteria devised in Section~\ref{sec:termination} appear to work well in practice, a more careful analysis of their effectiveness is lacking.  In addition, it would be of interest to better understand the behavior of the iterative scheme~(\ref{XupdateNewtonscaled}-\ref{EupdateNewtonscaled}) on ill-conditioned matrices.

\appendix
\section{Approximate Residuals} \label{sec:appendix}

In this section, we prove the validity of~(\ref{HWapprox}-\ref{HSapprox}).  Suppressing the subscript $k$ for the remainder of this section, our goal is to show that if
\begin{align}
\beta &= \frac{1}{2} \left( X^*X X^*E - X^* E X^*X \right), \label{betatilde} \\
\gamma &= (X^*E + E^* X) - \beta, \label{gammatilde}
\end{align}
then
\begin{align*}
\beta &= H\Omega - \Omega H + O(\|H^2-I\|^2 + \|H^2-I\| \|HS+SH\|), \\
\gamma &= HS+SH +  O(\|H^2-I\|^2 + \|H^2-I\| \|HS+SH\|). 
\end{align*}
Now since
\[
H^2 - I = 2(H-I) + (H-I)^2, 
\]
the norms of $H^2-I$ and $H-I$ are asymptotically equal, up to a factor of 2. Thus,
it is enough to show that
\begin{align}
\beta &= H\Omega - \Omega H + O(\|H-I\|^2 + \|H-I\| \|HS+SH\|), \label{HWapprox2} \\
\gamma &= HS+SH +  O(\|H-I\|^2 + \|H-I\| \|HS+SH\|). \label{HSapprox2}
\end{align}
The following lemma reduces this task to the verification of~(\ref{HWapprox2}).
\begin{lemma} \label{lemma:XE}
We have
\[
(H \Omega - \Omega H) + (H S + S H) = X^* E + E^* X.
\]
\end{lemma}
\begin{proof}
By~(\ref{relationXY}) and the equalities $H = H^*$, $U^*E = \Omega+S$, $\Omega^* = -\Omega$, and $S^*=S$, we have
\begin{align*}
X^* E + E^* X 
&= H U^* E + E^* U H \\
&= H (\Omega+S) + (\Omega+S)^* H \\
&= H (\Omega+S) + (-\Omega+S) H \\
&= (H \Omega - \Omega H) + (H S + S H).
\end{align*}
\end{proof}
It follows from the preceding lemma that if $\beta$ satisfies~(\ref{HWapprox2}), then $\gamma = (X^*E + E^* X) - \beta$ automatically satisfies~(\ref{HSapprox2}).

To prove~(\ref{HWapprox2}), we begin by noting a few useful relations.  
\begin{lemma} \label{lemma:Happrox}
For any $B \in \mathbb{C}^{n \times n}$, 
\begin{align*}
H(HB-BH) &= HB-BH + O(\|H-I\|^2), \\
(HB-BH)H &= HB-BH + O(\|H-I\|^2).
\end{align*}
\end{lemma}
\begin{proof}
These relations follow from the identities
\begin{align*}
H(HB-BH) = HB-BH - (H-I)B(H-I) + (H-I)^2 B, \\
(HB-BH)H = HB-BH + (H-I)B(H-I) - B(H-I)^2.
\end{align*}
\end{proof}

\begin{lemma} \label{lemma:XTXapprox}
We have
\[
X^*X = 2H-I + O(\|H-I\|^2).
\]
\end{lemma}
\begin{proof}
Use the identity 
\[
H^2 = 2H-I + (H-I)^2
\]
together with the fact that $X^*X = H U^*U H = H^2$.
\end{proof}

Now consider~(\ref{betatilde}).  Substituting $X^*X = 2H-I + O(\|H-I\|^2)$ and $X^*E = HU^*E = H(\Omega+S)$ gives, after simplification,
\begin{align*}
\beta
&= H \left( H(\Omega+S) - (\Omega+S)H \right) + O(\|H-I\|^2).
\end{align*}
Applying Lemma~\ref{lemma:Happrox} with $B = \Omega+S$ gives
\begin{align*}
\beta 
&= H(\Omega+S) - (\Omega+S)H + O(\|H-I\|^2) \\
&= (H\Omega - \Omega H) + (HS-SH) + O(\|H-I\|^2).
\end{align*}
We will finish the proof of~(\ref{HWapprox2}) by showing that
\[
HS-SH = O\left( \|H-I\|^2 + \|H-I\| \|HS+SH\| \right).
\]
Averaging the two equalities in Lemma~\ref{lemma:Happrox} with $B=S$ gives
\begin{align*}
HS-SH 
&= \frac{1}{2} \left[ H(HS-SH) + (HS-SH)H \right] + O( \|H-I\|^2) 
\end{align*}
Finally, an algebraic manipulation shows that the term in brackets above is equal to
\[
H(HS-SH) + (HS-SH)H = (H-I)(HS+SH) - (HS+SH)(H-I),
\]
and so it is of order $\|H-I\| \|HS+SH\|$.

\bibliographystyle{siamplain}
\bibliography{references}

\end{document}